\documentclass[11pt,amsmath]{article}

\usepackage{amsmath} 
\usepackage{amssymb}
\usepackage{amsfonts}
\usepackage{latexsym}
\usepackage{geometry}
\usepackage{graphicx}
\usepackage{subfigure}
\usepackage{cases}
\usepackage{enumitem}
\usepackage{titling}
\usepackage{bbm}
\usepackage{verbatim}
\usepackage[hang,flushmargin]{footmisc}
\allowdisplaybreaks
\usepackage[colorlinks=true,urlcolor=blue, citecolor=red,linkcolor=blue,
linktocpage,pdfpagelabels, bookmarksnumbered,bookmarksopen]
{hyperref}
 \usepackage{hyperref}
\usepackage[numbers,sort&compress]{natbib}

\usepackage{authblk}
\numberwithin{equation}{section}

\usepackage{amsmath, amsthm, amssymb, graphicx} 
\newtheorem{thm}{Theorem}[section]
\newtheorem{lem}[thm]{Lemma}

\newcommand{\HH}{\mathbb H}
\newcommand{\RR}{\mathbb R}
\newcommand{\lapH}{\Delta_{\HH^N}}
\newcommand{\lzero}{\lambda_0}

\usepackage{lipsum}

\title{Qualitative analysis of positive radial singular solutions on hyperbolic space}

\author{
Xia Huang\footnote{xhuang@cpde.ecnu.edu.cn},\ Yahui Jiang\footnote{52275500028@stu.ecnu.edu.cn}\ and Chunyi Zhao\footnote{Corresponding author, cyzhao@math.ecnu.edu.cn}\\
{\small School of Mathematical Sciences, Key Laboratory of MEA(Ministry of Education), \\ Shanghai Key Laboratory of PMMP, East China Normal University, Shanghai, China}
}

\date{}

\begin{document}
\maketitle

\begin{abstract}
We investigate positive radial solutions with an isolated nonremovable singularity for the semilinear elliptic equation
\begin{align*}
\lapH u+\lambda u+u^p=0
\qquad\text{in }\HH^N\setminus\{Q\},
\end{align*}
where $N\geq 3$, $p>1$, $\lambda\le \frac{(N-1)^2}{4}$, and $Q\in\HH^N$ is a prescribed pole.
Our purpose is to describe how the locally Euclidean singular behavior near ${\{Q}\}$ interacts with the genuinely hyperbolic dynamics at infinity, and how this interaction changes across the Serrin and Sobolev critical exponents.

For $\frac{N}{N-2}\le p<\frac{N+2}{N-2}$, we construct a family of positive radial singular solutions selecting the fast exponential mode at infinity. At the pole, these solutions exhibit the logarithmically corrected fundamental-solution profile when $p=\frac{N}{N-2}$, and the standard power-law profile when $\frac{N}{N-2}<p<\frac{N+2}{N-2}.$

 At the conformally invariant pair
$(p,\lambda)=(\frac{N+2}{N-2},\frac{N(N-2)}{4})$, we obtain an explicit singular solution and nonconstant Fowler-type solutions, and we classify all positive radial solutions with a nonremovable singularity via the positive periodic orbits of an associated autonomous equation.
In the supercritical regime $p>\frac{N+2}{N-2}$ with $\lambda\le\frac{ N(N-2)}{4}$, we prove existence and uniqueness of the global positive radial singular solution, derive its two-term local asymptotic expansion near the pole, and establish a sharp trichotomy for its behavior at infinity.
\end{abstract}

\medskip
\noindent\textbf{Keywords.}
Hyperbolic space; semilinear elliptic equation; positive radial solution; isolated singularity; asymptotic behavior; Fowler solution.

\smallskip
\noindent\textbf{2020 Mathematics Subject Classification.}
35J15, 35J61, 35A21, 35B40, 58J05.

\section{Introduction}

Let $\HH^N$ denote the $N$-dimensional hyperbolic space with $N\ge3$.  We consider the semilinear elliptic equation
\begin{equation}
\lapH u+\lambda u+u^p=0
\qquad\text{in }\HH^N\setminus\{Q\},
\label{eq1}
\end{equation}
where $p>1$, $\lambda\le (N-1)^2/4$, and $Q\in\HH^N$ is a prescribed point.
The point $Q$ is not a geometric singularity of the ambient space; but rather the location at which the solution is allowed to develop an isolated nonremovable singularity.

Semilinear elliptic equations on curved manifolds arise naturally at the interface of nonlinear analysis and conformal geometry, with the Yamabe equation providing one of the most prominent examples. In Euclidean space, the theory of isolated singularities for Lane--Emden and conformally invariant equations has been extensively developed and is by now rather well understood. Hyperbolic space introduces an additional feature that has no direct Euclidean counterpart: although the geometry is asymptotically Euclidean at an isolated pole, its negative curvature strongly affects the spectral structure of the Laplace--Beltrami operator and the behavior of solutions at infinity. Consequently, a singular solution on $\HH^N$ is governed by two different mechanisms: its leading behavior near the pole is essentially Euclidean, whereas its far-field behavior is genuinely hyperbolic. Understanding how these two mechanisms are matched globally is one of the main motivations of the present work.

\medskip

Since the isometry group of $\HH^N$ acts transitively, we may fix
$Q=(0,\ldots,0,1)$ in the hyperbolic model without loss of generality. For a solution $u=u(r)$ that is radial about $Q$, where $r=d_{\HH^N}(x,Q)$ denotes the hyperbolic distance from $x$ to $Q$, then \eqref{eq1} reduces to
\begin{equation}
u''+(N-1)\coth r\,u'+\lambda u+u^p=0,
\qquad r>0.
\label{eq:radial}
\end{equation}
The coefficient $\coth r$ reflects precisely the two geometric scales involved:
\begin{align*}
\coth r\sim \frac1r \quad (r\to0^+),
\qquad
\coth r\to1 \quad (r\to+\infty).
\end{align*}
Thus, near $r=0$, equation \eqref{eq:radial} is asymptotically the radial Lane-Emden equation in $\mathbb{R}^N$, while at infinity it approaches the autonomous equation
\begin{align*}
u''+(N-1) u'+\lambda u+u^p=0.
\end{align*}
This local to global transition is responsible for the different regimes described below.

The quantity $\lzero^2=\frac{(N-1)^2}{4}$ is the bottom of the spectrum of $-\lapH$. For $\lambda\le\lzero^2$, we define
\begin{equation}
\lzero=\frac{N-1}{2},
\qquad
\mu=\lzero-\sqrt{\lzero^2-\lambda},
\qquad
\nu=\lzero+\sqrt{\lzero^2-\lambda}.
\label{eq:indicial-roots}
\end{equation}
Then $\mu$ and $\nu$ are the roots of the characteristic equation
\begin{equation}
a^2+(1-N)a+\lambda=0.
\label{eigen}
\end{equation}
 Let $u$ be a positive solution of \eqref{eq:radial} that tends to zero at infinity. We say that $u$ decays rapidly if $e^{\lzero r}u(r)$ has a finite limit (possibly zero) as $r\to+\infty$; and decays slowly if $e^{\lzero r}u(r)\to+\infty$ as $r\to+\infty$. For the limiting linearized equation at infinity, the fast exponential mode is $e^{-\nu r}$, while the slow asymptotic behavior is determined by $e^{-\mu r}$. A central objective of this work is to identify the precise asymptotic behavior selected by  positive singular solutions in each parameter regime. 

\medskip

 Two exponents play a fundamental role, namely, the Serrin exponent $p_{\mathrm S}=\frac{N}{N-2}$ and the Sobolev critical exponent $ p_{\mathrm{Sob}}=\frac{N+2}{N-2}$. The local theory of isolated singularities for the corresponding Euclidean Lane-Emden equation has a long history. For $1<p<p_{\mathrm S}$, Lions \cite{L4} proved that any nonremovable singularity of a positive solution is of fundamental-solution type, with leading order $r^{2-N}$. At the Serrin exponent $p=p_{\mathrm S}$, Aviles \cite{A} showed that a nonnegative solution near an isolated singularity is either removable or exhibits a logarithmically corrected profile $r^{2-N}(-\ln r)^{-(N-2)/2}$. For $p_{\mathrm S}<p<p_{\mathrm{Sob}}$, Gidas and Spruck \cite{GS} established a complete classification, showing that a nonnegative solution is either regular or has the standard power-type singular behavior of order $r^{-2/(p-1)}$.

\medskip

Beyond local analysis, the Euclidean problem exhibits a different existence pattern, in particular, global positive radial singular solutions do not exist for $p\leq p_{\mathrm S}$. A systematic study of existence, uniqueness and asymptotic behavior in the Euclidean setting was carried out by Serrin and Zou \cite{SZ}

\medskip

At the Sobolev critical exponent, isolated singularities are deeply connected to conformal invariance.
The classical Yamabe equation 
\[
-\Delta u=u^{\frac{N+2}{N-2}}
\qquad\text{in }\RR^N\setminus\{0\}
\]
is a basic example. Caffarelli, Gidas, and Spruck \cite{CGS} proved the radial symmetry and complete classification of positive solutions with isolated singularities. Under the Emden--Fowler transformation, the problem reduces to an autonomous Hamiltonian system, its positive constant equilibrium gives the explicit singular solution, whereas its nonconstant periodic trajectories generate the classical Fowler (Delaunay-type) solutions. Related periodic singular structures have subsequently appeared in several conformally invariant elliptic equations, both local and nonlocal; see \cite{DPGW,GHWW,FK,JX,HLW}.

\medskip

Semilinear equations on hyperbolic space have also been studied from several viewpoints, but the global theory of singular solutions is less complete. Mancini and Sandeep \cite{MS} obtained existence, nonexistence and uniqueness results for positive finite-energy solutions in the range $1<p\le p_{\mathrm{Sob}}$, depending on the value of $\lambda$ and the dimension. Bandle and Kabeya \cite{BK,BK1} established existence and nonexistence results for both regular and singular radial solutions using perturbation arguments, variational methods, Pohozaev identities, and reductions to Matukuma-type equations. However, their work focuses primarily on existence criteria and does not provide sharp asymptotic descriptions of singular solutions, nor does it address the full classification of singularities at the critical exponent.

\medskip

For $\lambda=0$, Bonforte, Gazzola, Grillo, and V\'azquez \cite{BGGV} gave a detailed classification of radial solutions to the Emden-Fowler equation on hyperbolic space, including both finite- and infinite-energy solutions and their asymptotic behavior. At the conformally invariant pair $(p,\lambda)=(p_{\mathrm{Sob}},\frac{N(N-2)}{4})$, Li, Lu, and Wang \cite{LLW} established radial symmetry and explicit classification of positive solutions in $W_0^{1,2}(\HH^N)$, and further extended these results to GJMS equations with general nonlinearities.

\medskip

In the supercritical regime, Wu, Chen, Chern, and Kabeya \cite{WCCK} proved existence and uniqueness of a positive radial singular solution for the scalar-field equation in a high-dimensional, large-exponent regime above the Joseph--Lundgren threshold. For $\lambda=0$, He \cite{H1} showed that the standard power-type singular solution can be approximated by smooth classical solutions in the natural topology, in contrast to oscillatory radial solutions. More generally, Hasegawa \cite{H} established local existence, uniqueness, and asymptotic behavior of singular radial solutions for general supercritical nonlinearities on spherically symmetric Riemannian manifolds, including hyperbolic space.

\medskip

The purpose of the present paper is to complement and extend these results by simultaneously describing the singular behavior at the pole, the global continuation of the corresponding radial solution, and its asymptotic behavior at hyperbolic infinity. Our results reveal a sharp change in the global structure as $p$ passes through the Sobolev exponent. First, in the interval $p_{\mathrm S}\le p<p_{\mathrm{Sob}}$, we construct a family of singular solutions that selects the fast mode $e^{-\nu r}$ at infinity. Second, at the conformally invariant pair
$(p,\lambda)=(p_{\mathrm{Sob}},\frac{N(N-2)}{4})$, we identify an explicit singular profile and show that all nonremovable positive radial singularities are described by either constant or nonconstant positive periodic Fowler orbits.
Third, for $p>p_{\mathrm{Sob}}$, we construct a unique global positive singular solution for the full range $\lambda\le \frac{N(N-2)}{4}$ and determine its limiting behavior at infinity.

Before stating the main results, we recall a standard definition. 
A function $u\in C^2((0,\infty))$ is called a \emph{positive radial singular solution} if it solves \eqref{eq:radial}, remains positive for $r>0$, and satisfies
$u(r)\to+\infty$ as $r\to0^+$.

In what follows we set
\begin{equation}\label{aa}
	m=\frac{2}{p-1},
	\qquad
	L=\bigl(m(N-2-m)\bigr)^{1/(p-1)}.
\end{equation}
Our first result treats the range between the Serrin and Sobolev exponents.
\begin{thm}\label{thm1}
    Let $N\ge3$, $\frac{N}{N-2}\le p<\frac{N+2}{N-2}$, and
   $\lambda\le\frac{(N-1)^2}{4}$. Then \eqref{eq1} admits a family of positive radial singular solutions $u_s(r)$ such that, as $r\to0^+$
   \begin{align*}
   	u_s(r)\sim
   	\begin{cases}
   		\Bigl(\frac{N-2}{\sqrt{2}}\Bigr)^{N-2}
   		r^{2-N}
   		\bigl(\ln\frac1r\bigr)^{-\frac{N-2}{2}},\quad
   		&\text{if } p=\frac{N}{N-2},\\
   		Lr^{-m},
   		&\text{if } \frac{N}{N-2}<p<\frac{N+2}{N-2},
   	\end{cases}
   \end{align*}
   and
   \begin{align*}
   	u_s(r)\sim u_\infty e^{-\nu r}
   	\qquad\text{as }r\to+\infty,
   \end{align*}
   where $u_\infty>0$ and $\nu$ is defined in \eqref{eq:indicial-roots}.
\end{thm}

At the critical exponent, the value $\lambda=\frac{N(N-2)}{4}$ is distinguished by conformal invariance.
Besides an explicit solution, the equation possesses a family of Fowler-type singular solutions. Set \begin{align*}
	v_0=\biggl(\frac{N-2}{2}\biggr)^{(N-2)/2}.
	\end{align*}

\begin{thm}\label{thm2}
     Let $N\ge3$, $p=\frac{N+2}{N-2}$ and $\lambda=\frac{N(N-2)}{4}$. Then \eqref{eq1} has an explicit positive radial singular solution 
     \begin{align*}
         U_{s}(r)=\left(\frac{N-2}{2\sinh r}\right)^{\frac{N-2}{2}}.
     \end{align*}
     Moreover, \eqref{eq1} has a family of positive radial singular solutions $u_s$ satisfying 
     \begin{equation}\label{47}
         0<\liminf_{r\to0^+}(\sinh r)^{\frac{N-2}{2}}u_s(r)<v_0<\limsup_{r\to0^+}(\sinh r)^{\frac{N-2}{2}}u_s(r)<\left(\frac{N(N-2)}{4}\right)^{\frac{N-2}{4}}
     \end{equation}
     and $u_s(r)\sim\bar{u}_{\infty}e^{\frac{2-N}{2}r}$ as $r\to+\infty$ for some $\bar{u}_\infty>0$.
\end{thm}
The preceding periodic behavior gives a complete radial description of nonremovable singularities at the critical pair.

\begin{thm}\label{thm3}
	Assume the hypotheses of Theorem~\ref{thm2}, and let
	$u\in C^2(\HH^N\setminus\{Q\})$ be a positive radial solution of \eqref{eq1} with a nonremovable singularity at $Q$. Then either $u=U_s$, or there exist $b\in(0,v_0)$ and $T\in[0,T_b)$ such that
	\begin{align*}
	u(r)
	=
	(\sinh r)^{-\frac{N-2}{2}}
	v_b\left(\ln\tanh\frac r2+T\right),
	\qquad r>0,
	\end{align*}
	where $v_b$ denotes the positive nonconstant periodic solution of \eqref{1} introduced in Section~4, normalized by $v_b(0)=\inf_{\RR}v_b=b$, and $T_b$ denotes its least positive period.
\end{thm}
Finally, we consider the supercritical range.

\begin{thm}\label{thm4}
Let $N\ge3$, $p>\frac{N+2}{N-2}$ and $\lambda\le\frac{N(N-2)}{4}$. Then \eqref{eq1} admits a unique positive radial singular solution $u_s$, which satisfies
\begin{align*}
u_s(r)=L(\sinh r)^{-m}+\frac{L^p-(\lambda-m)L}{4N-4-6m}(\sinh r)^{2-m}+O\bigl((\sinh r)^{4-m}\bigr),
\qquad\text{as }r\to0^+.
\end{align*}
Furthermore, the following alternatives hold:
\begin{enumerate}[label=\textnormal{(\roman*)}]
\item If $\lambda<0$, then
\begin{align*}
u_s(r)\longrightarrow(-\lambda)^{1/(p-1)}
\qquad\text{as }r\to+\infty.
\end{align*}
\item If $\lambda=0$, then
\begin{align*}
u_s(r)\sim
\left(\frac{N-1}{p-1}\right)^{1/(p-1)}r^{-1/(p-1)}
\qquad\text{as }r\to+\infty.
\end{align*}
\item If $0<\lambda\le \frac{ N(N-2)}{4}$, then
\begin{align*}
u_s(r)\sim \bar u_\infty e^{-\mu r}
\qquad\text{as }r\to+\infty,
\end{align*}
where $\bar u_\infty>0$ and $\mu$ is defined in \eqref{eq:indicial-roots}.
\end{enumerate}
\end{thm}

The three parameter regimes require rather different arguments. To prove Theorem~\ref{thm1}, we introduce the inversion $s=\ln\coth\frac r2$, which exchanges the pole and hyperbolic infinity. After factoring out the fast decay mode,  we transform the problem into a weighted Euclidean radial equation via a positive solution of an associated linear equation. Existence and asymptotic results for that weighted equation can then be adapted to produce global solutions selecting the fast hyperbolic mode.
For Theorem~\ref{thm4}, we first construct a unique local solution near the singular profile $L(\sinh r)^{-m}$ by a fixed-point argument.
A Pohozaev-type identity prevents this solution from vanishing at a finite radius, while a dissipative energy controls its global continuation and asymptotic limits at infinity. For the critical pair, the conformal variable $t=\ln\tanh\frac r2$ reduces the equation to an autonomous Hamiltonian equation. Its constant equilibrium gives $U_s$, while its nonconstant periodic trajectories yield the oscillatory Fowler family and the classification in Theorem~\ref{thm3}.

\medskip

The remainder of the paper is organized as follows.
Section \ref{S2} proves Theorem~\ref{thm1} via the weighted Euclidean reduction.
Section \ref{S3} constructs the supercritical singular profile, derives the relevant Pohozaev identity, and proves Theorem~\ref{thm4}.
Section \ref{S4} treats the conformally invariant case and proves Theorems~\ref{thm2} and \ref{thm3} using phase-plane and energy arguments.
Finally, in the appendix, we collect some auxiliary results for the corresponding Euclidean weighted equation that are used in Section \ref{S2}.

\textit{Notation.} Throughout the paper, $C$ denotes a positive constant independent of $r$ that may change from line to line.

\section{Proof of Theorem \ref{thm1}} \label{S2}
In this section, we study the existence of positive singular solutions to \eqref{eq:radial} with fast decay at infinity, where $\frac{N}{N-2}\le p<\frac{N+2}{N-2}$ and $\lambda\le\lambda_{0}^2$. 

For $\lambda\le\lzero^2$, it follows from \cite[Lemmas~2.3 and~2.5]{BK} that if $u$ decays rapidly at infinity, then $u$ satisfies the precise asymptotic behavior
\begin{equation}\label{24}
	 e^{\nu r}u(r)\to u_\infty\qquad\text{as }r\to+\infty,
\end{equation}
for some $u_\infty>0$. We first establish the asymptotic expansions at infinity for positive solutions to \eqref{eq:radial} satisfying the fast-decay condition \eqref{24}.
\begin{lem}\label{lem4}
	Assume that $N \ge 3$, $p > 1$, and $\lambda \le \lzero^2$. Let $u$ be a positive solution to \eqref{eq:radial} satisfying the fast-decay condition \eqref{24}. Then, as $r \to +\infty$, we have the asymptotic expansions
	\begin{equation} \label{fast-expansion}
		u(r) = u_\infty e^{-\nu r} + O(e^{-\kappa r}), \qquad u'(r) = -\nu u_\infty e^{-\nu r} + O(e^{-\kappa r}),
	\end{equation}
	where $\kappa = \min\{\nu+2, p\nu\}$.
\end{lem}
\begin{proof}
First, we prove $u'(r) = O(e^{-\nu r})$ as $r\to+\infty$. Taking $R>0$ large enough and integrating \eqref{eq:radial} over $[R,r]$ for any $r>R$, we obtain 
	\begin{equation}\label{60}
		(\sinh r)^{N-1}u'(r)= (\sinh R)^{N-1}u'(R)-\int_{R}^{r}(\sinh \rho)^{N-1}(\lambda u+u^{p})\,d\rho,
	\end{equation}
	i.e.,
	\begin{equation}\label{26}
		u'(r)=\frac{(\sinh R)^{N-1}u'(R)}{(\sinh r)^{N-1}}-\frac{\int_{R}^{r}(\sinh \rho)^{N-1}(\lambda u+u^{p})\,d\rho}{(\sinh r)^{N-1}}.
	\end{equation} 
	By \eqref{24}, we see $(\sinh r)^{N-1}u=O(e^{\mu r})$ and $(\sinh r)^{N-1}u^{p}=O\left(e^{(N-1-p\nu)r}\right)$ as $r\to+\infty$. Let $\gamma= \max\{\mu, N-1-p\nu\}$. In fact, $\gamma=\mu$, since $N-1-p\nu<N-1-\nu=\mu$. We distinguish three cases:
	
	\textit{Case 1.} $\gamma > 0$. Passing to the limit in \eqref{26} and applying L'H\^{o}pital's rule, we get $u'(r) = O\left(e^{(\gamma - (N-1))r}\right)=O(e^{-\nu r})$ as $r\to+\infty$.
	
    \textit{Case 2.} $\gamma = 0$ (which occurs when $\lambda = 0$). The right-hand side of \eqref{60} converges at infinity. Hence, $\lim_{r\to+\infty} (\sinh r)^{N-1} u'(r) =\ell$ for some $\ell\in\RR$. Integrating \eqref{eq:radial} over $[r,+\infty)$ gives
	\begin{align*}
		\ell - (\sinh r)^{N-1} u'(r) = -\int_r^\infty (\sinh \rho)^{N-1} u^p\, d\rho = O\bigl(e^{(1-p)(N-1)r}\bigr).
	\end{align*}
	This implies $u'(r)=O\bigl(e^{(1-N)r}\bigr)=O(e^{-\nu r})$ as $r\to+\infty$.
	
	\textit{Case 3.} $\gamma < 0$ (which occurs when $\lambda < 0$). The right-hand side of \eqref{60} converges at infinity. Since $\lambda<0$ and $u(r)\to0$ as $r\to+\infty$, we see $\lambda u(r)+u^{p}(r)<0$ for $r$ large enough. We claim $u'(r)<0$ for $r$ large enough. Otherwise, there exists $r_0>0$ large enough such that $u'(r_0)\ge0$. By a similar identity to \eqref{60}, we have $u'(r)\ge0$ for $r\ge r_0$, which contradicts $u(r)\to0$ as $r\to+\infty$. Hence, there exists $\ell\le0$ such that $\lim_{r\to+\infty} (\sinh r)^{N-1} u'(r) =\ell$. Integrating \eqref{eq:radial} over $[r,+\infty)$ yields
	\begin{align*}
		\ell - (\sinh r)^{N-1} u'(r) = -\int_r^\infty (\sinh s)^{N-1}(\lambda u(s) + u^p(s)) \, ds = O(e^{\gamma r}).
	\end{align*}
	If $\ell<0$, the above identity implies $u'(r) \sim2^{N-1}\ell e^{-(N-1)r}$, and so $u(r)\sim\frac{2^{N-1}\ell}{1-N}e^{-(N-1)r}$ as $r\to+\infty$. Notice that $\nu > N-1$ when $\lambda < 0$. This contradicts the assumption $u(r) = O(e^{-\nu r})$. Thus, $\ell = 0$, and so $u'(r)=O\bigl(e^{(\gamma-(N-1))r}\bigr)=O(e^{-\nu r})$ as $r\to+\infty$.
	
 We now show a refined expansion at infinity. Define $q(r) = u(r) - u_\infty e^{-\nu r}$. Then $q(r) = o(e^{-\nu r})$ as $r\to+\infty$, and satisfies
	\begin{align*}
		q''+ (N-1)q' + \lambda q = \xi(r),
	\end{align*}
	where
	\begin{align*}
		\xi(r) = (N-1)(\coth r - 1)\bigl(\nu u_\infty e^{-\nu r} - q'(r)\bigr) - u^p(r).
	\end{align*}
	Since $\coth r - 1 = O(e^{-2r})$ and $q'(r) = u'(r) + \nu u_\infty e^{-\nu r} = O(e^{-\nu r})$ as $r\to+\infty$, we get $\xi(r) =O(e^{-\kappa r})$, where $\kappa = \min\{\nu+2, p\nu\}$. Since $p > 1$ and $\nu > 0$, we see $\kappa > \nu \ge \mu$.

		Suppose first that $\lambda < \lzero^2$, so that $\mu<\nu$. For $R>0$ large enough, we have 
	\begin{align*}
		q(r)=C_1 e^{-\nu r}+C_2e^{-\mu r}+\frac{e^{-\mu r}}{\nu-\mu}\int_{R}^r e^{\mu s}\xi(s)ds-\frac{e^{-\nu r}}{\nu-\mu}\int_{R}^r e^{\nu s}\xi(s)ds,\qquad r>R.
	\end{align*}
	Since  $\kappa>\nu>\mu$, we get 
	\begin{align*}
		C_2=-\frac{\int_{R}^{+\infty} e^{\mu s}\xi(s)ds}{\nu-\mu}.
	\end{align*}
	Consequently, for all sufficiently large $r$, 
	\begin{align*}
		\begin{aligned}
			q(r)&=C_1' e^{-\nu r}-\frac{e^{-\mu r}}{\nu-\mu}\int_{r}^{+\infty} e^{\mu s}\xi(s)ds+\frac{e^{-\nu r}}{\nu-\mu}\int_{r}^{+\infty} e^{\nu s}\xi(s)ds\\
			&=C_1' e^{-\nu r}+O(e^{-\kappa r}),
		\end{aligned}
	\end{align*}
	where $C_1'=C_1-\frac{\int_{R}^{+\infty}e^{\nu s}\xi(s)ds}{\nu-\mu}$. Thus, $C_1'=0$ since $q(r)=o(e^{-\nu r})$ as $r\to+\infty$.
	Moreover, 
	\begin{align*}
		q'(r)=\frac{\mu e^{-\mu r}}{\nu-\mu}\int_{r}^{+\infty} e^{\mu s}\xi(s)ds-\frac{\nu e^{-\nu r}}{\nu-\mu}\int_{r}^{+\infty} e^{\nu s}\xi(s)ds=O(e^{-\kappa r}),\qquad r\to+\infty.
	\end{align*}
	
		It remains to consider $\lambda=\lzero^2$. In this case $\mu=\nu=\lzero$. Set $z(r) = e^{\nu r}q(r)$. Then $z$ satisfies
	\begin{align*}
		z''(r) = e^{\nu r}\xi(r)=O\bigl(e^{(\nu-\kappa)r}\bigr)\qquad \text{as }r\to+\infty,
	\end{align*}
	which implies $\lim_{r\to+\infty}z'(r)=\ell_1$ since $\kappa>\nu$. Moreover, $q(r) = o(e^{-\nu r})$ implies $z(r) \to 0$ as $r \to +\infty$. Thus, we can deduce $\ell_1=0$. Hence, integrating $z''(r)$ twice over $[r,+\infty)$, we obtain
	\begin{align*}
		z'(r) &= -\int_r^\infty O\bigl(e^{(\nu-\kappa)\rho}\bigr) \, d\rho = O\bigl(e^{(\nu-\kappa)r}\bigr), \\
		z(r) &= -\int_r^\infty z'(\rho) \, d\rho = O\bigl(e^{(\nu-\kappa)r}\bigr).
	\end{align*}
This yields $q(r) = e^{-\nu r}z(r)=O\bigl(e^{-\kappa r}\bigr)$ as $r\to+\infty$. Furthermore, $q'(r) = e^{-\nu r}(z'(r) - \nu z(r))= O(e^{-\kappa r})$ as $r\to+\infty$.
	
	By the definition of $q$, we obtain
	\begin{align*}
		u(r) = u_\infty e^{-\nu r} + O(e^{-\kappa r}), \quad u'(r) = -\nu u_\infty e^{-\nu r} + O(e^{-\kappa r})\qquad\text{as }r\to+\infty,
	\end{align*}
	which completes the proof.
\end{proof}
Now we consider the following transformation
\begin{align*}
   s=\ln\coth\frac{r}{2},\qquad w(s)=s^{-\nu}u(r).
\end{align*}
Let $\alpha=2^{-\nu}u_\infty$, where $u_\infty$ is defined in \eqref{24}. Since $u$ satisfies \eqref{24}, we see $w(s)\to\alpha$ as $s\to0^+$. Moreover, 
\begin{align*}
    w'(s)=-s^{-\nu-1}\left(\nu u(r)+\frac{s}{\sinh s}u'(r)\right).
\end{align*}
Combining Lemma~\ref{lem4} with the expansion 
\begin{align*}
\frac{s}{\sinh s}=1-\frac{s^2}{6}+O(s^{4})\qquad\text{as }s\to0^+,
\end{align*}
we get 
\begin{align*}
\begin{aligned}
    w'(s)&=-s^{-(\nu+1)}\left[\nu\left(u_\infty e^{-\nu r}+O(e^{-\kappa r})\right)+\frac{s}{\sinh s}\left(-\nu u_\infty e^{-\nu r}+O(e^{-\kappa r})\right)\right]\\
    &=-s^{-(\nu+1)}\biggl[\nu u_\infty e^{-\nu r}\biggl(\frac{s^2}{6}+O\bigl(s^4\bigr)\biggr)+O(e^{-\kappa r})\biggr]\\
    &=-s^{-(\nu+1)}\left(O\bigl(s^{2+\nu}\bigr)+O(s^{\kappa})\right)\\
    &=O(s)+O\bigl(s^{\kappa-(\nu+1)}\bigr).
   \end{aligned} 
\end{align*}
Since $\nu+2-(\nu+1)=1$, $p\nu-(\nu+1)=(p-1)\nu-1\ge\frac{2\lzero}{N-2}-1>0$, we obtain $w'(s)\to0$ as $s\to0^+$.

 Substituting $w(s)$ into \eqref{eq:radial}, we see $w$ satisfies
\begin{equation}\label{KT}
\begin{cases}
w''+A(s)w'+B(s)w+D(s)w^p=0, & s>0,\\
w(0)=\alpha>0,\quad w'(0)=0,
\end{cases}
\end{equation}
where 
\begin{align*}
  A(s)&=\frac{2\nu}{s}+(2-N)\coth s,\\
B(s)&=\frac{\nu(\nu-1)}{s^2}+\frac{(2-N)\nu}{s}\coth s+\frac{\lambda}{(\sinh s)^2},
\end{align*}
and $D(s)=s^{\nu(p-1)}(\sinh s)^{-2}$. The asymptotic expansion near $0^+$ yields
\begin{align}
    A(s)&=\frac{2\nu+2-N}{s}+\frac{2-N}{3}s+O(s^3),\label{18}\\
      D(s)&=s^{\nu(p-1)-2}-\frac{s^{\nu(p-1)}}{3}+O(s^{\nu(p-1)+2}),\label{19}\\
     B(s)&=\frac{\nu(\nu-1)+(2-N)\nu+\lambda}{s^2}+\frac{(2-N)\nu-\lambda}{3}+O(s^2)=\frac{(2-N)\nu-\lambda}{3}+O(s^2),\label{20}
\end{align}
where the last equality follows from \eqref{eigen}. Define
\begin{align*}
I(s)=s^{2\nu}(\sinh s)^{2-N}.
\end{align*}
Equation \eqref{KT} can be written as
\begin{align*}
(Iw')'
=-I\bigl(Bw+Dw^p\bigr).
\end{align*}
Notice that $\nu(p-1)-2>-1$ since $p\ge\frac{N}{N-2}$. By a standard contraction argument, combining with \eqref{18}--\eqref{20}, we obtain that for each $\alpha>0$, \eqref{KT} admits a unique local positive solution $w_{\alpha}\in C([0,\varepsilon))\cap C^2((0,\varepsilon))$ for some $\varepsilon>0$. Moreover, we get $I(s)w'_\alpha(s)\to0$ as $s\to0^+$. By L'H\^{o}pital's rule, we get
\begin{align*}
	w_\alpha'(s)
=-\frac{1}{I(s)}
\int_0^s I(\rho)
\bigl[B(\rho)w_\alpha
+D(\rho)w_\alpha^p\bigr]\,d\rho\to0\qquad\text{as } s\to0^+.
\end{align*}

To prove the existence of positive singular solutions of \eqref{eq:radial} satisfying fast-decay condition, it suffices to show that \eqref{KT} admits global positive solutions that blow up at infinity at a certain rate. For this, we first study the existence and asymptotic behavior of positive solutions to the linear problem associated with \eqref{KT}
\begin{equation}\label{linear}
\begin{cases}
    (I\Phi')'=-IB\Phi, & s>0,\\
    \Phi(0)=1,\quad \Phi'(0)=0.
    \end{cases}
\end{equation}
\begin{lem}\label{lem6}
Let $N\ge3$ and $\lambda\le\lzero^2$. Then \eqref{linear} has a unique global positive solution $\Phi$. Moreover, there exists a constant $\bar{c}>0$ such that
\begin{equation}\label{53}
    \Phi(s)= 
    \begin{cases}    
   \bar{c}s^{-\nu}e^{(N-2)s}+O\left(s^{-\nu}\right)\quad&\text{if }N=3,\\
    \bar{c}s^{-\nu}e^{(N-2)s}+O\left(s^{1-\nu}\right)\quad&\text{if }N=4,\\
    \bar{c}s^{-\nu}e^{(N-2)s}+O\left(s^{-\nu}e^{(N-4)s}\right)\quad&\text{if }N\ge5,
    \end{cases}
\end{equation}
as $s\to+\infty$.
\end{lem}
\begin{proof}
	By a standard argument, together with \eqref{18} and \eqref{20}, we see that \eqref{linear} has a unique global solution $\Phi\in C([0,+\infty))\cap C^{2}((0,+\infty))$. Moreover, we can obtain $\Phi'(s)=O(s)$ and $\Phi(s)=1+O(s^2)$ as $s\to0^+$.
	
	We first prove the positivity of $\Phi$. Suppose first that
	$0\le\lambda\le\lzero^2$, so that \(\mu=\lambda/\nu\ge0\).
	For \(s>0\),
	\begin{align*}
		\frac{s^2B(s)}{\nu}=
		\nu-1+(2-N)s\coth s
		+\mu\frac{s^2}{\sinh^2s}<
		\nu-1+2-N+\mu=0,
	\end{align*}
since $s\coth s >1$ and $s^2\sinh^{-2}s<1$ for $s>0$. Since $I(s)\to0$ as $s\to0^+$ and $\Phi'(0)=0$, integrating \eqref{linear} over $[0,s]$, we have
	\begin{align*}
		I(s)\Phi'(s)=-\int_{0}^{s}I(\rho)B(\rho)\Phi(\rho)\,d\rho,
	\end{align*}
	which implies that $\Phi'(s)>0$ whenever $\Phi(s)>0$. Hence, $\Phi(s)>0$ for $s\ge0$. Otherwise, there exists $s_0>0$ such that $\Phi(s_0)=0$ and $\Phi(s)>0$ in $(0,s_0)$. It follows that $\Phi'(s)>0$ for $s\in(0,s_0)$, which contradicts $\Phi(s_0)=0$.
	
	Now suppose $\lambda<0$. Set $\overline{w}(r)=s^\nu\Phi(s)$. Then $\overline{w}$ satisfies
	\begin{equation}\label{61}
			\overline{w}''+(N-1)\coth r\,\overline{w}'+\lambda \overline{w}=0.
	\end{equation}
Moreover, we have
\begin{equation}\label{62}
\overline{w}(r)=2^\nu e^{-\nu r}(1+o(1)),\quad \overline{w}'(r)=-\nu2^\nu e^{-\nu r}(1+o(1))\qquad\text{as }r\to+\infty.
\end{equation}
This implies $\overline{w}(r)>0$ and $\overline{w}'(r)<0$ for $r$ large enough. It suffices to prove the positivity of $\overline{w}$. Suppose by contradiction that there exists $r_0>0$ such that $\overline{w}(r_0)=0$ and $\overline{w}>0$ on $(r_0,+\infty)$. Then $\overline{w}'(r_0)>0$. Equation \eqref{61} can be rewritten as
\begin{equation}\label{63}
	\bigl((\sinh r)^{N-1}\overline{w}'\bigr)'=-\lambda(\sinh r)^{N-1}\overline{w}.
	\end{equation}
Since $\lambda<0$, we have $\overline{w}'>0$ on $(r_0,+\infty)$, which is a contradiction. Thus, $\overline{w}(r)>0$ for $r\ge0$, i.e., $\Phi$ is positive on $(0,+\infty)$. 
	
	We next obtain the asymptotic behavior at infinity of $\Phi$. Set  $\widetilde{\Phi}(s)=s^{\nu}e^{(2-N)s}\Phi(s)$. Then $\widetilde{\Phi}$ satisfies 
	\begin{equation}\label{27}
		\widetilde{\Phi}''+[(N-2)+(N-2)(1-\coth s) \big]\widetilde{\Phi}'+\left[(N-2)^2(1-\coth s)+\frac{\lambda}{\sinh^2s} \right]\widetilde{\Phi}=0.
	\end{equation}
	Since $1-\coth s=O(e^{-2s})$, $(\sinh s)^{-2}=O(e^{-2s})$ as $s\to+\infty$, it follows from \cite[Theorem~8.1]{CL} that there exist constants $\bar{c}$ and $C_4$ such that 
	\begin{equation}\label{28}
		\widetilde{\Phi}(s)=\bar{c}(1+o(1))+C_4e^{-(N-2)s}(1+o(1))\qquad\text{as }s\to+\infty.
	\end{equation}
	Consequently, 
	\begin{align*}
		\Phi(s)=\bar{c}s^{-\nu}e^{(N-2)s}(1+o(1))+C_4s^{-\nu}(1+o(1))\qquad\text{as }s\to+\infty.
	\end{align*}
Positivity of $\Phi$ implies $\bar{c}\ge0$. For $\lambda\in[0,\lzero^2]$, $\Phi(s)\ge1$ for any $s\ge0$ guarantees $\bar{c}>0$. For $\lambda<0$, if $\bar{c}=0$, then $\overline{w}(r)\to C_4\ge0$ as $r\to0^+$. Integrating \eqref{63} and using L'H\^{o}pital's rule, we get $\overline{w}'(0)=0$. If $C_4=0$, by uniqueness, we obtain $\overline{w}\equiv0$, which is impossible. If $C_4>0$, it follows from \eqref{63} that $\overline{w}'(r)>0$ for $r>0$, which contradicts \eqref{62}. Thus, $\bar{c}>0$.

Finally, we give a precise asymptotic behavior at infinity of $\widetilde{\Phi}$. Set $\widetilde H(s)=
	\frac{e^{2(N-2)s}}{(\sinh s)^{N-2}}$. Equation \eqref{27} becomes
	\begin{equation}\label{29}
		(\widetilde H\widetilde\Phi')'
		=-\widetilde H\left[(N-2)^2(1-\coth s)
		+\frac{\lambda}{\sinh^2s}\right]\widetilde\Phi.
	\end{equation}
	By \eqref{28}, we have 
	\begin{equation}\label{30}
		\widetilde H(s)\left[(N-2)^2(1-\coth s)+\frac{\lambda}{\sinh^2s} \right]\widetilde{\Phi}(s)=O\bigl(e^{(N-4)s}\bigr)
		\quad\text{as }s\to+\infty.
	\end{equation}
	Integrating \eqref{29} from a fixed large \(S\) to \(s\), we obtain
	\begin{equation}\label{31}
		\widetilde H(s)\widetilde{\Phi}'(s)= \widetilde H(S)\widetilde{\Phi}'(S)-\int_{S}^{s}\widetilde H(\rho)\left[(N-2)^2(1-\coth\rho)+\frac{\lambda}{\sinh^2\rho} \right]\widetilde{\Phi}(\rho)\,d\rho.
	\end{equation}
	Integrating the above identity from $s$ to $+\infty$, we get
	\begin{align*}
		\widetilde{\Phi}(s)=
	\begin{cases}
		\bar{c}+O(e^{-s})\quad&\text{if }N=3,\\
		\bar{c}+O(se^{-2s})\quad&\text{if }N=4,\\
		\bar{c}+O(e^{-2s})\quad&\text{if }N\ge5.
		\end{cases}
	\end{align*}
	Thus, \eqref{53} holds.
\end{proof}

Define
\begin{align*}  T(s)=\int_{s}^{+\infty}\frac{d\rho}{I(\rho)\Phi^2(\rho)},\qquad s>0.
\end{align*}
It follows from Lemma~\ref{lem6} that $T\in C^2((0,+\infty))$ is well-defined and $T(s)>0$ for $s>0$. We introduce the following transformation 
\begin{align*}
	w(s)=\Phi(s)\widetilde{w}(\tau(s)),\qquad\tau(s)=[(N-2)T(s)]^{\frac{1}{2-N}}.
	\end{align*}
Then $\tau\in C^2((0,+\infty))$ is positive, and
\begin{align*}
	\tau'(s)=\frac{\tau^{N-1}(s)}{I(s)\Phi^2(s)}>0,\qquad\forall s>0.
\end{align*}
Set $d_\lambda=\sqrt{\lzero^2-\lambda}$. We discuss the asymptotic behavior of $\tau$ by considering two cases.

\textbf{Case 1.} $\lambda<\lzero^2$. Then $d_\lambda>0$. Observe that
\begin{align*}
I(s)\sim s^{1+2d_\lambda},\quad T(s)\sim\frac{s^{-2d_\lambda}}{2d_\lambda}\qquad\text{as }s\to0^+.
\end{align*}
By Lemma~\ref{lem6}, we have
\begin{alignat}{2}
	\tau(s)&\sim \left( \frac{N - 2}{2d_\lambda} \right)^{-\frac{1}{N-2}} s^{\frac{2d_\lambda}{N-2}},
	&\qquad& \text{as } s \to 0^+, \label{66}\\
	\tau(s)&\sim 2\bar{c}^{\frac{2}{N-2}} e^s,
	&\qquad& \text{as } s \to +\infty.\label{67}
\end{alignat}
 Hence, $\tau:(0,+\infty)\to(0,+\infty)$ is a bijection. Thus, $\widetilde{w}$ satisfies the following equation
\begin{align}\label{linear2}
	\begin{cases}
		\widetilde{w}''+\frac{N-1}{\tau}\widetilde{w}'+K(\tau)\widetilde{w}^{p}=0, & \tau>0,\\
		\widetilde{w}(0)=\alpha,\quad 	\widetilde{w}'(0)=0,
	\end{cases}
\end{align}
where 
\begin{align*}
	K(\tau)=\frac{D(s)I^{2}(s)\Phi^{p+3}(s)}{\tau^{2(N-1)}}.
\end{align*}
 It is easy to verify that $K\in C^1((0,+\infty))$ and $K(\tau)>0$ for $\tau>0$, since $D,I,\Phi\in C^2((0,+\infty))$ are positive on $(0,+\infty)$. By \eqref{66}, we get
\begin{equation}\label{54}
K(\tau)=\frac{s^{\nu(p+3)}\Phi^{p+3}(s)}{(\sinh s)^{2N-2}\tau^{2(N-1)}}=O\bigl(\tau^{l_1}\bigr)\qquad\text{as }\tau\to0^+,
\end{equation}
where $l_1=\frac{\nu(p-1)(N-2)}{2d_\lambda}-2$. Since $p\ge\frac{N}{N-2}$,
\begin{align*}
	l_1=\frac{\nu(p-1)(N-2)}{2d_\lambda}-2\ge\frac{\nu}{d_\lambda}-2>-1,
\end{align*}
which implies $\tau K\in L^{1}(0,1)$. 
 
 \textbf{Case 2.} $\lambda=\lzero^2$. Then $d_\lambda=\sqrt{\lzero^2-\lambda}=0$. In this case, there exists a constant $C_{T}$ such that
\begin{align*}
I(s)=s+O(s^3),\quad T(s)=\ln\frac{1}{s}+C_{T}+O(s^2),\qquad\text{as }s\to0^+.
\end{align*}
Using Lemma~\ref{lem6} again, we have
\begin{alignat}{2}
	\tau(s)&\sim\biggl((N-2)\ln\frac{1}{s}\biggr)^{\frac{1}{2-N}} ,
	&\qquad& \text{as } s \to 0^+, \label{68}\\
	\tau(s)&\sim 2\bar{c}^{\frac{2}{N-2}} e^s,
	&\qquad& \text{as } s \to +\infty.\label{69}
\end{alignat}
Hence, $\tau:(0,+\infty)\to(0,+\infty)$ is a bijection. Moreover, by the definition of $\tau$ and the above expansion of $T$, we obtain $s\sim e^{C_T} e^{-\frac{\tau^{2-N}}{N-2}}$ as $s\to0^+$. Hence,
\begin{equation}\label{65}
		K(\tau)\sim e^{\nu(p-1)C_{T}}\tau^{2-2N}e^{\frac{-\nu(p-1)}{N-2}\tau^{2-N}}\to0\qquad\text{as }\tau\to0^+,
\end{equation}
which implies $\tau K\in L^1(0,1)$.

In both cases, we conclude that $K$ satisfies $(\mathrm{K_1})$ (see the Appendix). Hence, a standard fixed-point argument shows that, for every $\alpha>0$, \eqref{linear2} admits a unique local positive solution $\widetilde{w}_{\alpha}\in C([0,\varepsilon))\cap C^2((0,\varepsilon))$ for some $\varepsilon>0$ in both cases. Moreover, we can obtain $\tau^{N-1}\widetilde{w}_\alpha'(\tau)\to0$ as $\tau\to0^+$. By \eqref{54} and \eqref{65}, L'H\^{o}pital's rule gives
\begin{align*}
	\lim_{\tau\to0^+}\widetilde w_\alpha'(\tau)=-\lim_{\tau\to0^+}\int_0^\tau
	(\rho/\tau)^{N-1}K(\rho)\widetilde w_\alpha^p(\rho)d\rho=0.
\end{align*}

Using the Euclidean weighted theory developed in \cite{L1,YY}, we establish the existence and asymptotic behavior at infinity of positive solutions to \eqref{KT} for $\frac{N}{N-2}<p<\frac{N+2}{N-2}$ and $p=\frac{N}{N-2}$, respectively.

\begin{lem}\label{lem7}
Let $N\ge3$, $\frac{N}{N-2}<p<\frac{N+2}{N-2}$, and $\lambda\le\lzero^2$. Then, there exists $\alpha_*>0$ such that for every
$\alpha\in(0,\alpha_*)$, \eqref{KT} admits a unique positive solution $w_\alpha$ satisfying
\begin{align*}
    w_\alpha(s)\sim\left[
    \frac{p(N-2)-N}{2(p-1)^2}
    \right]^{\frac1{p-1}}s^{-\nu}e^{\frac{2}{p-1}s}\qquad\text{as }s\to+\infty.
\end{align*}
\end{lem}
\begin{proof}
	By the preceding argument, for every $\alpha>0$, \eqref{linear2} admits a unique local positive solution. We first prove that there exists $\alpha_*>0$ such that for every $\alpha\in(0,\alpha_*)$, \eqref{linear2} admits a unique positive solution $\widetilde{w}_\alpha$ satisfying $\tau^{N-2}\widetilde{w}_\alpha(\tau)\to+\infty$ as $\tau\to+\infty$. Lemma~\ref{lem6} yields
	\begin{align*}
		T(s)\sim \frac{e^{(2-N)s}}{(N-2)2^{N-2}\bar{c}^2},\quad\tau(s)\sim 2\bar{c}^{\frac{2}{N-2}}e^s\qquad \text{as }s\to+\infty.
		\end{align*}
		 It follows that
	\begin{equation}\label{55}
		K(\tau)\sim k_{\infty}\tau^{l_2}\qquad\text{as }\tau\to+\infty,
	\end{equation}
	where $k_{\infty}=\left(2^{N-2}\bar{c}\right)^{\frac{N+2}{N-2}-p}$ and $l_2=p(N-2)-(N+2)$. Since $\frac{N}{N-2}<p<\frac{N+2}{N-2}$, we have $l_2\in(-2,0)$ and $l_2=p(N-2)-(N+2)<\frac{p(N-2)-(N+2)}{2}$. Furthermore, we have
\begin{equation}\label{56}
\tau^{N-1-(N-2)p}K(\tau)\sim k_\infty\tau^{-3}\in L^1(1,+\infty),
\end{equation}
which implies $K$ satisfies $(\mathrm{K_2})$. Hence, using \eqref{55} and Lemma~\ref{lem9}, we see $\liminf_{\tau\to+\infty}G(\tau)\le\limsup_{\tau\to+\infty}G(\tau)<0$, where $G$ is defined in Appendix. It follows from Lemma~\ref{lem8} that there exists $\alpha_{*}$ such that for every $\alpha\in(0,\alpha_*)$, \eqref{linear2} admits a unique global positive solution $\widetilde{w}_\alpha$ satisfying $\tau^{N-2}\widetilde{w}_\alpha(\tau)\to+\infty$ as $\tau\to+\infty$.

Now we prove the asymptotic behavior at infinity of $\widetilde{w}_{\alpha}$ for $\alpha\in(0,\alpha_*)$. A direct computation shows 
 \begin{align*}
 \frac{D'(s)}{D(s)}=\frac{\nu(p-1)}{s}-2\coth s,\quad \frac{I'(s)}{I(s)}=\frac{2\nu}{s}-(N-2)\coth s,\quad \frac{\tau'(s)}{\tau(s)}=\frac{1}{2-N}\frac{T'(s)}{T(s)},
 \end{align*}
  and
     \begin{equation}\label{39}
         \frac{\Phi'(s)}{\Phi(s)}=\frac{-\nu}{s}+N-2+\frac{\widetilde{\Phi}'(s)}{\widetilde{\Phi}(s)}.
     \end{equation}  
Hence, 
 \begin{equation}\label{eq:K-ln-derivative}
 \begin{aligned}
     \frac{\frac{d}{ds}\left(\tau^{-l_2}K(\tau)\right)}{\tau^{-l_2}K(\tau)}&=\frac{D'(s)}{D(s)}+2\frac{I'(s)}{I(s)}+(p+3)\frac{\Phi'(s)}{\Phi(s)}-\bigl(p(N-2)+N-4\bigr)\frac{\tau'(s)}{\tau(s)}\\
     &=(p+3)(N-2)-2(N-1)\coth s+\left(p+\frac{N-4}{N-2}\right)\frac{T'(s)}{T(s)}+(p+3)\frac{\widetilde{\Phi}'(s)}{\widetilde{\Phi}(s)}.
      \end{aligned}
 \end{equation}
From the proof of Lemma~\ref{lem6}, we see 
\begin{equation}\label{40}
    \frac{\widetilde{\Phi}'(s)}{\widetilde{\Phi}(s)}=O(e^{-\sigma s}),\quad \widetilde{\Phi}'(s)=O(e^{-\sigma s})\qquad\text{as }s\to+\infty,
\end{equation}
where
\begin{align*}
    \sigma=\begin{cases}
        1\quad &\text{if } N = 3, \\
        2-\epsilon\quad &\text{if } N=4,\\
        2\quad &\text{if } N\ge5,
    \end{cases}
\end{align*}
and $\epsilon>0$ is small.
Set $\widetilde{G}(s)=I(s)\Phi^{2}(s)$. Then 
\begin{align*}
    \frac{\widetilde{G}'(s)}{\widetilde{G}(s)}=\frac{I'(s)}{I(s)}+2\frac{\Phi'(s)}{\Phi(s)},
\end{align*}
together with \eqref{39} and \eqref{40}, we see
\begin{align*}
      \frac{\widetilde{G}'(s)}{\widetilde{G}(s)}=N-2+O(e^{-\sigma s})\qquad\text{as }s\to+\infty.
\end{align*}
This also yields
\begin{equation}\label{41}
    \frac{\widetilde{G}(s)}{\widetilde{G}'(s)}=\frac{1}{N-2}+O(e^{-\sigma s})\qquad\text{as }s\to+\infty.
\end{equation}
In addition,
\begin{align*}
\begin{aligned}
    \frac{d}{ds}\frac{\widetilde{G}(s)}{\widetilde{G}'(s)}&=-\biggl(\frac{\widetilde{G}(s)}{\widetilde{G}'(s)}\biggr)^2 \frac{d}{ds} \frac{\widetilde{G}'(s)}{\widetilde{G}(s)}\\
    &=-\biggl(\frac{\widetilde{G}(s)}{\widetilde{G}'(s)}\biggr)^2\biggl[(N-2)\operatorname{csch}^2s+2\frac{\widetilde{\Phi}''(s)}{\widetilde{\Phi}(s)}-2\biggl(\frac{\widetilde{\Phi}'(s)}{\widetilde{\Phi}(s)}\biggr)^2\biggr].
    \end{aligned}
\end{align*}
By \eqref{27} and \eqref{40}, we know 
\begin{align*}
    \frac{\widetilde{\Phi}''(s)}{\widetilde{\Phi}(s)}=O(e^{-\bar{\sigma} s})\qquad\text{as }s\to+\infty,
\end{align*}
where $\bar{\sigma}=\min\{\sigma, 2\}$. Combining this with \eqref{41}, we obtain
\begin{equation}\label{42}
     \frac{d}{ds}\frac{\widetilde{G}(s)}{\widetilde{G}'(s)}=O(e^{-\bar{\sigma}s})\qquad\text{as }s\to+\infty.
\end{equation}
Therefore, from \eqref{41} and \eqref{42}, there holds
      \begin{align*}
      \begin{aligned}
T(s)&=\int_{s}^{+\infty}\frac{dt}{\widetilde{G}(t)}=\int_s^{+\infty} \biggl( \frac{\widetilde{G}(t)}{\widetilde{G}'(t)}\biggr) \frac{\widetilde{G}'(t)}{\widetilde{G}^2(t)} dt \\
    &= \biggl[ -\frac{1}{\widetilde{G}(t)} \frac{\widetilde{G}(t)}{\widetilde{G}'(t)} \biggr]_s^{+\infty} + \int_s^{+\infty} \frac{1}{\widetilde{G}(t)} \frac{d}{dt}\frac{\widetilde{G}(t)}{\widetilde{G}'(t)}dt\\
    &=\frac{1}{(N-2)\widetilde{G}(s)}\bigl(1+O(e^{-\sigma s})\bigr)+O\bigl(e^{(2-N-\bar{\sigma})s}\bigr)\\
    &=\frac{1}{(N-2)\widetilde{G}(s)}\bigl(1+O(e^{-\bar{\sigma}s})\bigr)\qquad\text{as }s\to+\infty.
    \end{aligned}
      \end{align*} 
  This also yields
  \begin{equation}\label{43}
      \frac{T'(s)}{T(s)}=-\frac{1}{\widetilde{G}(s)T(s)}=2-N+O(e^{-\bar{\sigma}s})\qquad\text{as }s\to+\infty,
  \end{equation}
  and 
\begin{equation}\label{44}
    \frac{\tau'(s)}{\tau(s)}=1+O(e^{-\bar{\sigma}s})\qquad\text{as }s\to+\infty.
\end{equation} 
  Substituting \eqref{55}, \eqref{40} and \eqref{43} into \eqref{eq:K-ln-derivative} yields
  \begin{align*}
        \frac{d}{ds}\bigl(\tau(s)^{-l_2}K(\tau)\bigr)=O(e^{-2s})+O(e^{-\bar{\sigma}s})+O(e^{-\sigma s})=O(e^{-\bar{\sigma}s})\qquad\text{as }s\to+\infty.
  \end{align*}
Combining this with \eqref{44}, we obtain
  \begin{align*}
      \frac{d}{d\tau}\bigl(\tau(s)^{-l_2}K(\tau)\bigr)= \frac{d}{ds}\tau(s)^{-l_2}K(\tau)\frac{ds}{d\tau}=O\bigl(e^{-(1+\bar{\sigma})s}\bigr)\qquad\text{as }s\to+\infty,
  \end{align*}
 which implies $\left|\frac{d}{d\tau}\left(\tau^{-l_2}K(\tau)\right)\right|\in L^1(1,+\infty)$. Moreover, since $\frac{N}{N-2}<p<\frac{N+2}{N-2}$, we see $0<\frac{2+l_2}{p-1}<\frac{N-2}{2}$. It follows from Lemma~\ref{lem11} that 
 \begin{align*}
     \lim_{\tau\to+\infty}\tau^{\frac{2+l_2}{p-1}}\widetilde{w}_\alpha(\tau)=\begin{cases}
         \biggl[\frac{\frac{2+l_2}{p-1}\bigl(N-2-\frac{2+l_2}{p-1}\bigr)}{k_\infty}\biggr]^{\frac{1}{p-1}}\qquad\text{or}\\
         0.
     \end{cases}
 \end{align*}
Since $\tau^{N-2}\widetilde{w}_\alpha(\tau)\to+\infty$ as $\tau\to+\infty$, applying Lemma~\ref{lem11} again yields that for every $\alpha\in(0,\alpha_*)$,
\begin{equation}\label{38}
    \widetilde{w}_\alpha(\tau)\sim\Biggl[\frac{\frac{2+l_2}{p-1}\bigl(N-2-\frac{2+l_2}{p-1}\bigr)}{k_\infty}\Biggr]^{\frac{1}{p-1}}\tau^{-\frac{2+l_2}{p-1}}\qquad\text{as }\tau\to+\infty.
\end{equation}

Applying Lemma~\ref{lem6} again, we obtain that for every $\alpha\in(0,\alpha_*)$, \eqref{KT} admits a unique positive solution $w_\alpha(s)=\Phi(s)\widetilde{w}_\alpha(\tau(s))>0$ for $s\ge0$. By \eqref{53}, \eqref{67} and \eqref{38}, we have 
\begin{align*}
	w_\alpha(s)=\Phi(s)\widetilde{w}_\alpha(\tau(s))\sim
\left[
\frac{p(N-2)-N}{2(p-1)^2}
\right]^{\frac1{p-1}}
s^{-\nu}e^{\frac{2}{p-1}s}\qquad\text{as }s\to+\infty,
\end{align*}
which completes the proof. 
   \end{proof}

\begin{lem}\label{lem12}
	Let $N\ge3$, $p=\frac{N}{N-2}$, and $\lambda\le\lzero^2$. Then there exists $\alpha^{*}>0$ such that for every
	$\alpha\in(0,\alpha^{*})$, \eqref{KT} admits a unique positive solution $w_\alpha$ satisfying
	\begin{align*}
		w_\alpha(s)\sim
		\biggl[\frac{(N-2)^2}{8}\biggr]^{\frac{N-2}{2}}
		s^{-\nu-\frac{N-2}{2}}e^{(N-2)s}
		\qquad\text{as }s\to+\infty.
	\end{align*}
\end{lem}
\begin{proof}
	As in the proof of Lemma~\ref{lem7}, for every $\alpha>0$, \eqref{linear2} admits a unique local positive solution $\widetilde w_\alpha$. Moreover, \eqref{56} still holds, and so $K$ satisfies $(\mathrm{K_2})$. Notice that, when $p=\frac{N}{N-2}$, 
	\begin{align}\label{70}
		K(\tau)\sim k_\infty\tau^{-2}
		\qquad\text{as }\tau\to+\infty,
	\end{align}
	where $k_\infty=4\bar c^{\frac{2}{N-2}}$. Lemma \ref{lem8} and Lemma \ref{lem9} imply that there exists
	$\alpha^{*}>0$ such that, for every
	$\alpha\in(0,\alpha^{*})$, the corresponding solution
	$\widetilde w_\alpha$ extends to a global positive solution satisfying $\tau^{N-2}\widetilde w_\alpha(\tau)\rightarrow+\infty$ as $\tau\to+\infty$.

	The estimates \eqref{39}--\eqref{44} are independent of the strict inequality
	$p>\frac{N}{N-2}$ and remain valid for $p=\frac{N}{N-2}$. Thus, we also have $\left|\frac{d}{d\tau}\bigl(\tau^2K(\tau)\bigr)\right|\in L^1(1,+\infty)$. Applying Lemma~\ref{lem13} yields
	\begin{align*}
		\lim_{\tau\to+\infty}
		(\ln\tau)^{\frac{1}{p-1}}
		\widetilde w_\alpha(\tau)
		=\begin{cases}
		\biggl[
		\frac{N-2}{(p-1)k_\infty}
		\biggr]^{\frac{1}{p-1}}
		\qquad\text{or}\\
		0.
		\end{cases}
	\end{align*} 
	Since $\tau^{N-2}\widetilde w_\alpha(\tau)\rightarrow+\infty$ as $\tau\to+\infty$, Lemma~\ref{lem13} again yields that for every $\alpha\in(0,\alpha^{*})$, 
	\begin{align*}
		\widetilde w_\alpha(\tau)\sim\biggl[\frac{N-2}{(p-1)k_\infty}\biggr]^{\frac{1}{p-1}}(\ln\tau)^{-\frac{1}{p-1}}\qquad\text{as }\tau\to+\infty.
	\end{align*}
	By \eqref{67} and \eqref{69}, $\ln\tau(s)\sim s$ as $s\to+\infty$. It follows from Lemma~\ref{lem6} that for every $\alpha\in(0,\alpha^{*})$, \eqref{KT} admits a unique positive solution $w_\alpha$ satisfying
	\begin{align*}
		w_\alpha(s)=\Phi(s)\widetilde{w}_\alpha(\tau(s))\sim\biggl[\frac{(N-2)^2}{8}\biggr]^{\frac{N-2}{2}}s^{-\nu-\frac{N-2}{2}}e^{(N-2)s}
		\qquad\text{as }s\to+\infty.
	\end{align*}
	This completes the proof.
\end{proof}

\textbf{Proof of Theorem~\ref{thm1}.} 
Observe that $s\sim2e^{-r}$ as $r\to+\infty$, while $s\sim\ln\frac{2}{r}$ as $r\to0^+$. When $p=\frac{N}{N-2}$, Lemma~\ref{lem12} yields that for every $\alpha\in(0,\alpha^*)$, there exists a positive singular solution $u_s$ of
\eqref{eq:radial} satisfying
\begin{align*}
	u_s(r)=s^\nu w_\alpha(s)\sim
	\biggl(\frac{N-2}{\sqrt{2}}\biggr)^{\frac{N-2}{2}}r^{2-N}\biggl(\ln\frac1r\biggr)^{-\frac{N-2}{2}}\qquad\text{as }r\to0^+,
	\end{align*}
	and
	\begin{equation}\label{71}
	u_s(r)=s^\nu w_\alpha(s)\sim2^\nu\alpha e^{-\nu r}=u_\infty e^{-\nu r}\qquad\text{as }r\to+\infty,
\end{equation}
where $u_\infty=2^\nu\alpha\in(0,2^\nu\alpha^{*})$.

In the interval $\frac{N}{N-2}<p<\frac{N+2}{N-2}$, applying Lemma~\ref{lem7}, we obtain that for every
$\alpha\in(0,\alpha_*)$, \eqref{eq:radial} admits a positive singular solution $u_s$ satisfying
\begin{align*}
	u_s(r)=s^{\nu}w_\alpha(s)\sim Lr^{-\frac{2}{p-1}}\qquad\text{as }r\to0^+,
\end{align*}
and \eqref{71} with $u_\infty\in(0,2^\nu\alpha_{*})$. The proof of Theorem~\ref{thm1} is completed.\qed

\section{Proof of Theorem \ref{thm4}}\label{S3}
In this section, we study the existence of positive singular solutions to \eqref{eq:radial} in the supercritical case and present the proof of Theorem~\ref{thm4}.

Set
\begin{align*}
	z=\ln\sinh r,\qquad W(z)=(\sinh r)^{m}u(r),
	\end{align*}
where $m$ is given in (\ref{aa}).
Then $W$ satisfies
\begin{equation}\label{LE}
    W''+f(z)W'-m(N-2-m)W+\frac{(\lambda-m)e^{2z}}{1+e^{2z}}W+\frac{1}{1+e^{2z}}W^p=0,
\end{equation}
where 
\begin{align*}
    f(z)=N-2-2m+\frac{e^{2z}}{1+e^{2z}}.
\end{align*}
The following lemma describes the asymptotic profile of singular solutions to \eqref{eq:radial} near the origin. This lemma extends the result obtained in \cite[Lemma~3.1]{WCCK}.
\begin{lem}\label{lem10}
    Let $N\ge3$, $p>\frac{N+2}{N-2}$, and $\lambda\le\lzero^2$. Let $u(r)$ be a positive singular solution of \eqref{eq:radial}. Then 
    \begin{equation}\label{57}
        \lim_{r\to0^+}(\sinh r)^{m}u(r)=L.
    \end{equation}
Here $L$ is given by (\ref{aa}).
\end{lem}
\begin{proof}
	Since $p>\frac{N+2}{N-2}$ and $u(r)\to+\infty$ as $r\to0^+$, we have $\lambda u+u^{p}\ge\frac{1}{2}u^{p}$ for $r>0$ sufficiently small. Moreover, $\coth r\sim\frac{1}{r}$ as $r\rightarrow0^+$. Similarly to the proof of Theorem~2.1 in \cite{NS1}, there exist $r_0>0$ and $C>0$ such that $u(r)\le C r^{-m}$ for $r\in(0,r_0)$. Then $W(z) = (\sinh r)^m u(r)$ is bounded on $(-\infty, z_0)$, where $z_0 = \ln\sinh r_0$. 
	 
	 We claim that $W'$ is bounded on $(-\infty,z_0)$. Suppose by contradiction that there exists a sequence $z_n\to-\infty$ such that $|W'(z_n)|\to+\infty$ as $n\to+\infty$. Observe that $0<N-2-2m\le f(z)\le N-1-2m$ since $p>\frac{N+2}{N-2}$. Since $W$ is
	 bounded, the variation-of-constants formula yields that there exists $C>0$ such that
\begin{align*}
	|W'(z)|\ge e^{-(N-1-2m)(z-z_n)}|W'(z_n)|-C(z-z_n),\qquad z\ge z_n.
	\end{align*}
	 Choose $\delta>0$ sufficiently small so that $e^{-(N-1-2m)\delta}\ge\frac34$. Since $|W'(z_n)|\to+\infty$, we have $|W'(z_n)|\ge4C\delta$ for all sufficiently large $n$. Hence, $|W'(z)|\ge\frac12|W'(z_n)|>0$ for every $z\in[z_n,z_n+\delta]$. It follows that
\begin{align*}
	 \begin{aligned}
	 	|W(z_n+\delta)-W(z_n)|=
	 	\biggl|
	 	\int_{z_n}^{z_n+\delta}W'(z)\,dz
	 	\biggr|=
	 	\int_{z_n}^{z_n+\delta}|W'(z)|\,dz\ge
	 	\frac{\delta}{2}|W'(z_n)|\rightarrow+\infty,
	 \end{aligned}
	 \end{align*}
	 which contradicts the boundedness of $W$. Hence, $W'$ is bounded on $(-\infty,z_0]$.
	 
	 We now prove \eqref{57}. Let $f_1(z) = \frac{e^{2z}}{1+e^{2z}}$ and define the energy functional
	 \begin{align*}
	 	E(z) = \frac{1}{2}W'^2(z) - \frac{L^{p-1}}{2}W^2(z) + \frac{\lambda-m}{2}f_1(z)W^2(z) + \frac{1-f_1(z)}{p+1}W^{p+1}(z).
	 \end{align*}
	A simple calculation gives
	 \begin{align*}
	 	E'(z) = -f(z)W'^2(z) + f_1'(z)\biggl[ \frac{\lambda-m}{2}W^2(z) - \frac{1}{p+1}W^{p+1}(z) \biggr].
	 \end{align*}
	 Since $W$, $W'$ and $f_1$ are bounded on $(-\infty, z_0)$, $E$ is bounded on $(-\infty, z_0)$. Notice that $f_1'(z) = \frac{2e^{2z}}{(1+e^{2z})^2}\in L^1(-\infty, z_0)$. Integrating $E'(z)$ over $[z,z_0]$ gives for any $z\le z_0$,
	 \begin{align*}
	 \begin{aligned}
	 	(N-2-2m)\int_z^{z_0}W'^2(\rho)\,d\rho&\le
	 	\int_z^{z_0}f(\rho)W'^2(\rho)\,d\rho\\
	 	&\le E(z)-E(z_0)
	 	+C\int_z^{z_0}f_1'(\rho)\,d\rho
	 \end{aligned}
	 \end{align*}
	 for some $C>0$. This implies $W'\in L^2(-\infty,z_0)$. Equation \eqref{LE} and the boundedness of $W$ and $W'$ imply that
	 $W''$ is bounded,  and so $W'$ is uniformly continuous. It follows that $\lim_{z\to-\infty}W'(z)=0$. Furthermore, $W'\in L^2(-\infty,z_0)$, $f_1'\in L^1(-\infty,z_0)$ and the boundedness of $W$ yield $E'\in L^1(-\infty,z_0)$. This also implies that there exists $E_{-\infty}\in\RR$ such that $\lim_{z\to-\infty} E(z) = E_{-\infty}$. Since $W' \to 0$ and $f_1 \to 0$ as $z \to -\infty$, we get
	 \begin{equation}\label{64}
	 	\lim_{z\to-\infty} \biggl( -\frac{L^{p-1}}{2}W^2(z) + \frac{1}{p+1}W^{p+1}(z) \biggr) = E_{-\infty}.
	 \end{equation}
	 We claim that $\lim_{z\to-\infty} W(z)$ exists. Otherwise, there exist $W^-<W^+$ such that
	 \begin{align*}
	 W^-=\liminf_{z\to-\infty}W(z)
	 <\limsup_{z\to-\infty}W(z)=W^+.
	 \end{align*}
	 The continuity of $W$ implies that every value in
	 $(W^-,W^+)$ is attained by $W$ along some sequence tending to $-\infty$. This contradicts \eqref{64}, since the left-hand side of \eqref{64} is not constant on $(W^-,W^+)$. Thus, $\lim_{z\to-\infty} W(z) = W_{-\infty}$. Passing to the limit in \eqref{LE}, we obtain
	 \begin{align*}
	 	\lim_{z\to-\infty}W''(z)=m(N-2-m)W_{-\infty} - W_{-\infty}^p=\ell_{2}.
	 \end{align*}
	  We claim $\ell_2=0$. Indeed, if $\ell_2\neq0$, then there exists $z_1<z_0$ such that $W''$ has a fixed sign and $|W''(z)|\ge\frac{|\ell_2|}{2}$ for $z\le z_1$. Consequently,
	 \begin{align*}
	 |W'(z_1)-W'(z)|
	 =\biggl|\int_z^{z_1} W''(\rho)\,d\rho\biggr|
	 \ge\frac{|\ell_2|}{2}(z_1-z)
	 \rightarrow+\infty\qquad\text{as }z\to-\infty,
	 \end{align*}
	which contradicts $\lim_{z\to-\infty}W'(z)=0$. Thus $\ell_2=0$, and so $W_{-\infty}\in\{0,L\}$.
	 
It remains to exclude $W_{-\infty}=0$. Suppose that $W_{-\infty}=0$. Set $\Psi=(W,W')^{T}$. Then $\Psi$ satisfies 
\begin{align*}
		\Psi'=\begin{pmatrix}
			0&1\\
			m(N-2-m)&-(N-2-2m)
		\end{pmatrix}\Psi+\begin{pmatrix}
		0&0\\
		f_{1}(W^{p-1}-\lambda+m)&-f_1
		\end{pmatrix}\Psi+\begin{pmatrix}
		0\\
		-W^p
		\end{pmatrix}.
		\end{align*}
		The eigenvalues of the above constant matrix are $m>0$ and $-(N-2-m)<0$, and an eigenvector corresponding to $m$ is $\varphi_1=(1,m)^{T}$. Since $f_1(z)=O(e^{2z})$ as $z\to-\infty$ and $W$ is bounded, we have $f_1,f_1W^{p-1}\in L^{1}(-\infty,z_0)$. By \cite[Lemma~2.1(iii)]{BK}, there exists a constant $C\neq0$ such that $\Psi(z)=Ce^{mz}(1+o(1))\varphi_1$ as $z\to-\infty$. The positivity of $W$ guarantees $C>0$. Hence,
		\begin{align*}
			u(r) = (\sinh r)^{-m}W(z) = e^{-mz}W(z) \to C \qquad \text{as } r \to 0^+,
		\end{align*}
		which implies $u$ is bounded near the origin. This contradicts the nonremovable singularity of $u$. Hence, $W_{-\infty} = L$, and so \eqref{57} holds.
\end{proof}
We next prove the existence and uniqueness of a local positive singular solution of \eqref{eq:radial} using the contraction mapping theorem.
\begin{lem}\label{lem2}
    Let $N\ge3$, $p>\frac{N+2}{N-2}$, and $\lambda\le\lzero^2$. Then \eqref{eq:radial} admits a unique local
    positive singular solution near the origin. Moreover, there exists $r_0>0$ small enough such that $u(r)=L(\sinh r)^{-m}(1+\theta(r))$ solves \eqref{eq:radial} on $(0,r_0)$, where $\theta(r)\in C^2(0,r_0)$ satisfies
     \begin{equation}\label{58}
    		\theta(r)=\frac{L^{p-1}+m-\lambda}{4N-4-6m}(\sinh r)^2+O\left((\sinh r)^{4}\right)\qquad\text{as }r\to0^+,
    \end{equation}
    and
    \begin{equation}\label{59}
    		\theta'(r)=\frac{L^{p-1}+m-\lambda}{4N-4-6m}\sinh 2r+O\left((\sinh r)^{3}\right)\qquad\text{as }r\to0^+.
    \end{equation}
\end{lem}
\begin{proof}
	First, we prove the existence of local positive singular solutions of \eqref{LE}. By Lemma~\ref{lem10}, we seek a local positive singular solution of \eqref{LE} in the form 
\begin{align*}
    W(z)=L+x(z):= L + L \theta(r).
\end{align*}
Substituting $W(z)$ into \eqref{LE}, we have
\begin{align*}
   x''+f(z)x'-m(N-2-m)(L+x)+\frac{(\lambda-m)e^{2z}}{1+e^{2z}}(L+x)+\frac{1}{1+e^{2z}}(L+x)^p=0.
\end{align*}
We construct a solution satisfying $x(z)\to0$ as $z\to-\infty$. Set $\zeta(z)=(x(z),y(z))^{T}$, where $y(z)=x'(z)$. Then $\zeta$ satisfies
\begin{align*}
    \frac{d}{dz}\zeta=A\zeta+J(\zeta,z),
\end{align*}
where
\begin{align*}
	A=\begin{pmatrix}
        0 & 1\\ -2(N-2-m)& -(N-2-2m)  
    \end{pmatrix},\quad J(\zeta,z)=\begin{pmatrix}
    0 \\ g(x,y,z)
    \end{pmatrix},\end{align*}
with 
\begin{align*}
g(x,y,z)=-\frac{e^{2z}}{1+e^{2z}}y+\frac{e^{2z}}{1+e^{2z}}\bigl[L^p-(\lambda-m)L+(pL^{p-1}-\lambda+m)x\bigr]-\frac{R(x)}{1+e^{2z}},
\end{align*}
and $R(x)=(L+x)^p-L^p-pL^{p-1}x$.  
    
Let $0<\varepsilon_0<\min\bigl\{\frac{L}{2},1\bigr\}$. Define $X=C((-\infty,Z_0],\RR^2)$ equipped with the norm 
\begin{align*}
	\|\zeta\|_X=
\|x\|_{L^\infty(-\infty,Z_0]}
+
\|y\|_{L^\infty(-\infty,Z_0]},
\end{align*}
and $B_\varepsilon=
\left\{
\zeta\in X:
\|\zeta\|_X\le\varepsilon
\right\}$, where \(0<\varepsilon\le\varepsilon_0\) and \(Z_0<0\) are to be determined. For any
$\zeta_i=(x_i,y_i)^{T}\in B_\varepsilon$, $i=1,2$, and
\(z\le Z_0\), we have
\begin{align*}
	|g(x_1,y_1,z)-g(x_2,y_2,z)|\le e^{2z}|y_1-y_2|+|pL^{p-1}-\lambda+m|e^{2z}|x_1-x_2|+C\varepsilon|x_1-x_2|,
\end{align*}
which implies
\begin{equation}\label{9}
\begin{aligned}
	|J(\zeta_1(z),z)-J(\zeta_2(z),z)|
	&\le
	C\left(e^{2z}+\varepsilon\right)
	|\zeta_1(z)-\zeta_2(z)|\\
	&\le
	C\left(e^{2Z_0}+\varepsilon\right)
	|\zeta_1(z)-\zeta_2(z)|.
\end{aligned}
\end{equation}
Define the operator $\mathcal{F}$ by
\begin{align*}
\mathcal F(\zeta)(z)
:=
\int_{-\infty}^{z}
e^{(z-\rho)A}J(\zeta(\rho),\rho)\,d\rho,
\qquad z\le Z_0.
\end{align*}
It suffices to prove the existence of a fixed point of $\mathcal F$ in $B_\varepsilon$. Since $p>\frac{N+2}{N-2}$, we have $N-2-2m>0$, which implies that all eigenvalues of $A$ have strictly negative real parts. Thus there exist constants $C>0$ and $\bar{\gamma}<0$ such that $\|e^{tA}\|_{\RR^2}\le C e^{\bar{\gamma}t}$. Together with \eqref{9}, for any $z\le Z_0$, we have
\begin{align*}
	\bigl\|\mathcal{F}(\zeta_1)(z)
	-\mathcal{F}(\zeta_2)(z)\bigr\|_{\RR^2}&\le
	\int_{-\infty}^{z}
	\left\|
	e^{(z-\rho)A}
	\bigl(
	J(\zeta_1(\rho),\rho)
	-
	J(\zeta_2(\rho),\rho)
	\bigr)
	\right\|_{\RR^2}\,d\rho
	\\
	&\quad\le
	C\bigl(e^{2Z_0}+\varepsilon\bigr)
	\int_{-\infty}^{z}
	e^{\bar{\gamma}(z-\rho)}
	\|\zeta_1(\rho)-\zeta_2(\rho)\|_{\RR^2}\,d\rho
	\\
	&\quad\le
	\frac{C\bigl(e^{2Z_0}+\varepsilon\bigr)}
	{|\bar{\gamma}|}
	\|\zeta_1-\zeta_2\|_X.
\end{align*}
This yields 
\begin{align*}
	\|\mathcal{F}(\zeta_1)-\mathcal{F}(\zeta_2)\|_{X} \le \frac{C\bigl(e^{2Z_0}+\varepsilon\bigr)}{|\bar{\gamma}|}\|\zeta_{1}-\zeta_{2}\|_{X}.
\end{align*}
We now choose $\varepsilon\in(0,\varepsilon_0]$ sufficiently
small so that $\frac{C\varepsilon}{|\bar{\gamma}|}\le\frac14$. Since 
\begin{align*}
	|g(0,0,z)|\le\frac{e^{2z}}{1+e^{2z}}(L^p+|\lambda-m|L),
	\end{align*}
which implies
\begin{align*}
\|\mathcal F(0)\|_X\le
C\sup_{z\le Z_0}
\int_{-\infty}^{z}
e^{\bar\gamma(z-\rho)}e^{2\rho}\,d\rho\le
Ce^{2Z_0}.
\end{align*}
Hence, for the fixed $\varepsilon$, we can choose $Z_0<0$ sufficiently
negative such that $\frac{Ce^{2Z_0}}{|\bar{\gamma}|}<\frac{1}{4}$ and $\|\mathcal F(0)\|_X\le\frac{\varepsilon}{2}$. Then
\begin{align*}
\|\mathcal F(\zeta_1)-\mathcal F(\zeta_2)\|_X
\le\frac12\|\zeta_1-\zeta_2\|_X.
\end{align*}
Moreover, for every \(\zeta\in B_\varepsilon\),
\begin{align*}
	\|\mathcal F(\zeta)\|_X\le\|\mathcal F(\zeta)-\mathcal F(0)\|_X
	+\|\mathcal F(0)\|_X\le
	\frac12\varepsilon+\frac12\varepsilon
	=\varepsilon.
\end{align*}
By the contraction mapping theorem, $\mathcal{F}$ admits a unique fixed point $\zeta^*=(x^*, y^*)^T$ in $B_{\varepsilon}$. Since $|x^*(z)|\le\varepsilon\le \frac L2$, $W^*(z)=L+x^*(z)>0$ for $z\le Z_0$.  

Second, we prove the uniqueness of a local positive singular solution of \eqref{eq:radial}. Let $\widetilde u_s$ be any positive singular solution of \eqref{eq:radial} and set $\widetilde W(z)=(\sinh r)^m\widetilde u_s(r)$. By Lemma~\ref{lem10} and its proof, any positive singular solution
of \eqref{eq:radial} satisfies
\begin{align*}
\widetilde W(z)-L\to0,\quad \widetilde W'(z)\to0
\qquad\text{as }z\to-\infty.
\end{align*}
Consequently, there exists $Z_1 \le Z_0$ sufficiently negative such that $\widetilde\zeta=\bigl(\widetilde W-L,\widetilde W'\bigr)^T\in B_\varepsilon$ for $z\le Z_1$. Since $\operatorname{Re}\gamma_{\pm}<0$ and $\widetilde\zeta(z)\to0$ as $z\to-\infty$, the variation of constants formula yields that
\begin{align*}
\widetilde\zeta(z)
=\int_{-\infty}^{z}
e^{(z-\rho)A}J(\widetilde\zeta(\rho),\rho)\,d\rho,
\qquad z\le Z_1.
\end{align*}
This means that $\widetilde\zeta$ is a fixed point of the operator $\mathcal{F}$ restricted to $(-\infty,Z_1]$.
Moreover, the restriction of $\zeta^*$ to
$(-\infty,Z_1]$ is a fixed point of the operator $\mathcal{F}$. By the uniqueness of the fixed point of $\mathcal{F}$ in $B_{\varepsilon}$, $\widetilde W\equiv W^*$ on $(-\infty,Z_1]$, which proves the local uniqueness.

Finally, we derive the precise asymptotic behavior of $\theta(r)$ near the origin. For any $Z\le Z_0$, define $M(Z):=\sup_{z\le Z}|\zeta^{*}(z)|$. For any $z\le Z$, we have
\begin{align*}
	\begin{aligned}
	|\zeta^{*}(z)|&=\left|\int_{-\infty}^ze^{(z-\rho)A}J(\zeta^{*}(\rho),\rho)\,d\rho\right|\\
	&\le C\int_{-\infty}^ze^{\bar{\gamma}(z-\rho)}\left(e^{2\rho}+\varepsilon|\zeta^{*}(\rho)|\right)d\rho\\
	&\le\frac{Ce^{2Z}}{2-\bar{\gamma}}+\frac{C\varepsilon}{|\bar{\gamma}|}M(Z).
	\end{aligned}
\end{align*}
Since $\frac{C\varepsilon}{|\bar{\gamma}|}\le\frac{1}{4}$, the inequality guarantees that $\zeta^{*}(z)=O\bigl(e^{2z}\bigr)$ as $z\to-\infty$. Observing that 
\begin{align*}
	\frac{e^{2z}}{1+e^{2z}}=e^{2z}+O\bigl(e^{4z}\bigr)\qquad\text{as } z\to-\infty,
	\end{align*}
combining with $\zeta^{*}(z)=O\bigl(e^{2z}\bigr)$ as $z\to-\infty$, we have
\begin{align*}
g(x^{*},y^{*},z)=\left(L^p-(\lambda-m)L\right)e^{2z}+O\bigl(e^{4z}\bigr)\qquad\text{as }z\to-\infty,
\end{align*}
and 
\begin{align*}
J(\zeta^{*}(z),z)=\left(L^p-(\lambda-m)L\right)e^{2z}e_2+O\bigl(e^{4z}\bigr)\qquad\text{as }z\to-\infty,
\end{align*}
where $e_2=(0,1)^T$. Since $\operatorname{Re}\gamma_{\pm}<0$, we get
\begin{align*}
	\begin{aligned}
		\zeta^{*}(z)&=\int_{-\infty}^ze^{(z-\rho)A}J(\zeta^{*}(\rho),\rho)\,d\rho\\
		&=\left(L^p-(\lambda-m)L\right)\int_{-\infty}^{z}
		e^{(z-\rho)A}e_2e^{2\rho}\,d\rho+O\bigl(e^{4z}\bigr)\\
		&=\left(L^p-(\lambda-m)L\right)e^{2z}\int_{0}^{+\infty}e^{(A-2I_2)\rho}e_2\,d\rho+O\bigl(e^{4z}\bigr)\\
		&=\left(L^p-(\lambda-m)L\right)(2I_2-A)^{-1}e_2\,e^{2z}+O\bigl(e^{4z}\bigr)\qquad\text{as }z\to-\infty,
	\end{aligned}
	\end{align*}
where $I_2$ is the $2\times2$ identity matrix. A direct computation shows 
\begin{align*}
(2I_2-A)^{-1}e_2
=(4N-4-6m)^{-1}(1,2)^{T}.
\end{align*}
Thus, we obtain
\begin{align*}
	\begin{aligned}
x^{*}(z)&=\frac{L^p-(\lambda-m)L}{4N-4-6m}e^{2z}+O\bigl(e^{4z}\bigr)\qquad\text{as }z\to-\infty,\\
y^{*}(z)&=\frac{L^p-(\lambda-m)L}{2N-2-3m}e^{2z}+O\bigl(e^{4z}\bigr)\qquad\text{as }z\to-\infty,
	\end{aligned}
\end{align*}
which yields \eqref{58} and \eqref{59}.
\end{proof}

The following lemma rules out the possibility that the local positive singular solution obtained in Lemma~\ref{lem2} vanishes at a finite point.
\begin{lem}\label{lem3}
    Assume that $N\ge3$, $p>\frac{N+2}{N-2}$, and $\lambda\le\frac{N(N-2)}{4}$. Let $u_s$ denote the maximal continuation of the local positive singular solution of \eqref{eq:radial} obtained in Lemma~\ref{lem2}. Then there is no $R\in(0,\infty)$ such that $u_{s}(r)>0$ for $r\in(0,R)$ and $u_{s}(R)=0$.
\end{lem}
\begin{proof}
	Arguing by contradiction, suppose that there exists $R\in(0,\infty)$ such that $u_{s}(r)>0$ in $(0,R)$ and $u_{s}(R)=0$. Set $h(r)=(\sinh r)^{\lzero}u_s(r)$. Then $h$ satisfies 
\begin{equation}\label{v}
    h''-a(r)h+b(r)h^p = 0, 
\end{equation}
where $a(r) = \lzero^2 - \lambda +\frac{(N-3)\lzero}{2}(\operatorname{csch}r)^2$ and $b(r)=(\sinh r)^{-\lzero(p-1)}$.

Let $g\in C^{\infty}([0,+\infty))$. Multiplying \eqref{v} by $h'g$ and integrating over $[\varepsilon,R]$, we obtain  
\begin{equation}\label{51}
\frac{1}{2} \int_\varepsilon^R g'h'^2\,dr = \biggl(\frac{h'^2g}{2} - \frac{agh^2}{2}+\frac{bgh^{p+1}}{p+1}\biggr)\biggr|_\varepsilon^R + \frac{1}{2}\int_\varepsilon^R (ag)'h^2\,dr- \frac{1}{p+1} \int_\varepsilon^R (bg)'h^{p+1}\,dr.
\end{equation}
Multiplying \eqref{v} by $g'h$ and integrating over $[\varepsilon,R]$, we have   
\begin{equation}\label{52}
\frac{1}{2} \int_\varepsilon^R g'h'^2\,dr= \biggl(\frac{g'hh'}{2}- \frac{h^2 g''}{4}\biggr)\biggr|_\varepsilon^R + \int_\varepsilon^R \biggl( \frac{g'''}{4} - \frac{ag'}{2} \biggr) h^2\,dr + \int_\varepsilon^R \frac{g'b}{2} h^{p+1}\,dr.
\end{equation}
Combining \eqref{51} and \eqref{52}, we get
	\begin{equation}\label{pohozaev}
	\begin{aligned}
		&\biggl(\frac{h'^2g}{2}-\frac{agh^2}{2}+\frac{bgh^{p+1}}{p+1}-\frac{ g'hh'}{2}+\frac{h^2 g''}{4}\biggr)\biggr|_\varepsilon^R \\
		&\quad = \int_\varepsilon^R \biggl[ \biggl( \frac{g'''}{4} - \frac{a'g}{2}-ag'\biggr)h^2 + \biggl(\frac{g'b}{2}+\frac{(bg)'}{p+1}\biggr) h^{p+1} \biggr]\, dr.
	\end{aligned}
\end{equation}
Now take $g(r)=\sinh r$. From Lemma~\ref{lem2}, we see  $u_s(r) \sim L r^{-m}$ and $u_s'(r) \sim -mL r^{-m-1}$ as $r\to0^+$. It follows that near the origin,
\begin{align*}
		h(r)\sim Lr^{\frac{N-1}{2}-\frac{2}{p-1}},\qquad h'(r)\sim(\lzero-m)Lr^{\frac{N-1}{2}-1-\frac{2}{p-1}}.
\end{align*}
Since $p>\frac{N+2}{N-2}$, it is easy to check that, as $r\to0^+$
	\begin{equation}\label{48}
	b(r)g(r)h^{p+1}(r) \to 0, \quad a(r)h^2(r)g(r) \to 0, \quad h^2(r)g''(r) \to 0,
\end{equation}
and similarly for the derivative cross-terms,
\begin{equation}\label{49}
	h'(r)h(r) \to 0, \quad h'^2(r)g(r) \to 0, \quad g'(r)h(r)h'(r) \to 0.
\end{equation}
Moreover, $u_{s}(R)=0$ and $u_{s}(r)>0$ on $(0,R)$ yield $u_{s}'(R)<0$. By the definition of $h$, we see
\begin{equation}\label{50}
	h(R)=0 \quad \text{and} \quad 0<|h'(R)|<\infty.
\end{equation}
Letting $\varepsilon\to0^+$ in \eqref{pohozaev} and invoking \eqref{48}--\eqref{50}, the identity reduces to 
	\begin{align*}
		\begin{aligned}
			& \frac{1}{2}h'^2(R)\sinh R + \int_{0}^{R} \left(\frac{N(N-2)}{4}-\lambda\right)\cosh r \, h^2 \, dr \\
			&\quad = \frac{(2-N)p+(N+2)}{2(p+1)} \int_{0}^{R} \cosh r (\sinh r)^{-\lzero(p-1)} h^{p+1} \, dr.
		\end{aligned}
	\end{align*}
Since $\lambda\le\frac{N(N-2)}{4}$ and $p>\frac{N+2}{N-2}$, the left-hand side is strictly positive, whereas the right-hand side is strictly negative. This is a contradiction, which completes the proof.
\end{proof}

\begin{lem}\label{Pohozaev1}
Assume that $N\ge3$ and $p>1$. Let $u\in C^2([r_1,r_2])$ be a positive solution of \eqref{eq:radial} in $(r_1,r_2)$, where $0<r_1<r_2$. Then the following Pohozaev-type identity holds
	\begin{equation}\label{pohozaev-identity}
	\begin{aligned}
		&\biggl[ (\sinh r)^N \biggl( \frac{u'^2}{2}+\frac{N-2}{2}\coth r\,uu' +\left(\frac{\lambda}{2}-\frac{N-2}{4}\right)u^2 + \frac{u^{p+1}}{p+1} \biggr) \biggr]\biggr|_{r_1}^{r_2} \\
		&\quad = \int_{r_1}^{r_2}(\sinh r)^{N-1}\cosh r \biggl[ \left(\lambda-\frac{N(N-2)}{4}\right)u^2 + \left(\frac{N}{p+1}-\frac{N-2}{2}\right)u^{p+1} \biggr] \, dr.
	\end{aligned}
\end{equation}
\end{lem}
\begin{proof}
Multiplying \eqref{eq:radial} by $(\sinh r)^N u'$ and integrating over $[r_1,r_2]$, we obtain
\begin{align*}
\int_{r_1}^{r_2}\bigl[(\sinh r)^N u''u' + (N-1)(\sinh r)^N\coth r \, u'^2 + (\sinh r)^N(\lambda u+u^p)u'\bigr]\,dr
=0.
\end{align*}
Using
\begin{align*}
\left(\frac{(\sinh r)^Nu'^2}{2}
\right)'=(\sinh r)^Nu''u'+\frac{N(\sinh r)^{N-1}\cosh r\,u'^2}{2},
\end{align*}
and integrating by parts, we get
\begin{equation}\label{7}
	\begin{aligned}
		&(\sinh r)^N \biggl( \frac{u'^2}{2}+\frac{\lambda}{2}u^2 +\frac{u^{p+1}}{p+1} \biggr)\biggr|_{r_1}^{r_2} \\
		&\quad = \int_{r_1}^{r_2} (\sinh r)^{N-1}\cosh r \biggl[ -\frac{N-2}{2}u'^2+N\biggl( \frac{\lambda}{2}u^2 +\frac{1}{p+1}u^{p+1} \biggr) \biggr] \, dr.
	\end{aligned}
\end{equation}
Multiplying \eqref{eq:radial} by $\cosh r(\sinh r)^{N-1} u$ and integrating over $[r_1,r_2]$ yields
\begin{equation}\label{8}
\begin{aligned}
		\int_{r_1}^{r_2} \cosh r(\sinh r)^{N-1}u'^2 \, dr
   &=\int_{r_1}^{r_2}\cosh r[((\sinh r)^{N-1}uu')'+(\sinh r)^{N-1}(\lambda u^2+u^{p+1})]dr\\
    &=(\sinh r)^{N-1}\biggl( \cosh r\,uu'-\frac{\sinh r}{2}u^2 \biggr)\biggr|_{r_1}^{r_2}\\
    &\quad +\int_{r_1}^{r_2} (\sinh r)^{N-1}\cosh r \biggl( \frac{N+2\lambda}{2}u^2+u^{p+1} \biggr) \, dr.
    \end{aligned}
\end{equation}
Substituting \eqref{8} into \eqref{7}, we obtain \eqref{pohozaev-identity}.
\end{proof}
We now prove Theorem~\ref{thm4}.

\textbf{Proof of Theorem~\ref{thm4}.} The proof is divided into three steps.

\textbf{Step 1.} We prove the existence and uniqueness of a positive singular solution to \eqref{eq:radial}. Lemma~\ref{lem2} shows that if $p>\frac{N+2}{N-2}$ and $\lambda\le\lzero^2$, \eqref{eq:radial} has the unique positive local singular solution $u_s$ satisfying
\begin{align*}
u_{s}(r)=L(\sinh r)^{-m}+\frac{L^p-(\lambda-m)L}{4N-4-6m}(\sinh r)^{2-m}+O\bigl((\sinh r)^{4-m}\bigr)\qquad \text{as }r\to0^+.
\end{align*}
By Lemma~\ref{lem3}, $u_{s}$ is positive whenever it exists for $\lambda\le\frac{N(N-2)}{4}$. Consider the energy functional
    \begin{align*}
        \mathcal{E}_{u}(r)=\frac{u'^2(r)}{2}+\frac{\lambda u^2}{2}+\frac{u^{p+1}}{p+1}.
    \end{align*}
    If $u$ is a positive solution to \eqref{eq:radial}, then $\mathcal{E}_u'(r)=-(N-1)\coth r\,u'^2(r)\le0$. For any $\lambda\le\frac{N(N-2)}{4}$, there exists a constant $C_0$ such that 
    \begin{align*}
    	\frac{\lambda u^2}{2}+\frac{u^{p+1}}{p+1}\ge C_0,\qquad\text{for }u>0,
    	\end{align*}
    and 
    \begin{align*}
    	\frac{\lambda u^2}{2}+\frac{u^{p+1}}{p+1}\to+\infty\qquad\text{as }u\to+\infty.
    	\end{align*} 
    Thus, for any $r_{*}\in(0,r_0)$,
    \begin{align*}
    	C_0\le\mathcal{E}_{u_{s}}(r)\le\mathcal{E}_{u_{s}}(r_*)<\infty,\qquad r\ge r_{*},
    \end{align*}
 which implies that $u_s$ and $u_s'$ are bounded for $r\ge r_*$. By a standard ODE argument, \eqref{eq:radial} admits the unique positive singular solution $u_{s}$ if $\lambda\le\frac{N(N-2)}{4}$ and $p>\frac{N+2}{N-2}$.

 \textbf{Step 2.} We study the limiting behavior of $u_s$ as $r\to+\infty$.
 By \eqref{eq:radial}, $u_s''$ is bounded for $r\ge r_*$, and so $u_s'$ is uniformly continuous for $r\ge r_*$. Since $\mathcal{E}_{u_s}$ is non-increasing and bounded below in $[r_*,+\infty)$, there exists $\ell_3\in\RR$ such that $\mathcal{E}_{u_s}(r)\to\ell_3$ as $r\to+\infty$. This yields $u_{s}'\in L^2(r_*,+\infty)$. Thus, $u_{s}'(r)\to0$ as $r\to+\infty$, and so 
 \begin{align*}
 	\frac{\lambda u_{s}^2(r)}{2}+\frac{u_{s}^{p+1}(r)}{p+1}\to\ell_3\qquad\text{as }r\to+\infty.
 \end{align*}
 
\textit{Case 1.} $\lambda\in\bigl[0,\frac{N(N-2)}{4}\bigr]$. It is easy to see that there exists $\bar{\ell}\ge0$ such that $u_{s}(r)\to\bar{\ell}$ as $r\to+\infty$. By \eqref{eq:radial}, $u_{s}''(r)\to-\lambda\bar{\ell}-\bar{\ell}^{p}$ as $r\to+\infty$. Since $u_{s}'(r)\to0$, we deduce $-\lambda\bar{\ell}-\bar{\ell}^{p}=0$, which yields $\bar{\ell}=0$.
 
\textit{Case 2.} $\lambda<0$. 
Suppose by contradiction that $u_s(r)$ does not converge as $r\to+\infty$. Since $u_{s}$ is continuous and bounded in $(r_{*},+\infty)$, there exist two sequences $r_{n}^{+}\to+\infty$ and $r_{n}^{-}\to+\infty$ such that
\begin{align*}
	\lim_{n\to+\infty}u_s(r^{+}_n)=u_1>u_2=	\lim_{n\to+\infty}u_s(r^{-}_n),
\end{align*}
where $u_{1}=\limsup_{r\to+\infty}u_{s}(r)$ and $u_2=\liminf_{r\to+\infty}u_{s}(r)$ both satisfy $\frac{\lambda u_i^2}{2}+\frac{u_i^{p+1}}{p+1}=\ell_3$. Choose $u_3\in(u_2,u_1)$ such that $\frac{\lambda u_3^2}{2}+\frac{u_3^{p+1}}{p+1}\neq\ell_3$. By the Intermediate Value Theorem, there exists a sequence $r_n\to+\infty$ such that $u_{s}(r_n)=u_3$ for every $n\in\mathbb{N}$, which is a contradiction. Thus, $u_{s}(r)\to\bar{\ell}$ as $r\to+\infty$. Similar to Case 1, we obtain $-\lambda\bar{\ell}-\bar{\ell}^{p}=0$, which implies $\bar{\ell}\in\{0,(-\lambda)^{\frac{1}{p-1}}\}$.

\textbf{Step 3.} We investigate the asymptotic behavior of $u_{s}$ at infinity by distinguishing three cases. 

Before proceeding, we observe from Lemma~\ref{lem2} that $u_{s}(r)\sim Lr^{-m}$ and $u_{s}'(r)\sim-mLr^{-m-1}$ as $r\to0^+$. Since $p>\frac{N+2}{N-2}$, as $r\to0^+$,
\begin{align*}
	(\sinh r)^N u_s^{p+1}(r) \to 0, \quad (\sinh r)^N u_s^2(r) \to 0, \quad (\sinh r)^N u_s'^2(r) \to 0,
\end{align*}
and $(\sinh r)^{N}\coth r\,u_s(r)u_s'(r) \to 0$. Consequently, all the boundary terms at $r_1\to0^+$ in the Pohozaev identity \eqref{pohozaev-identity} vanish in the three cases below.

\textit{Case 1.} $\lambda<0$. We prove $u_s(r)\to(-\lambda)^{\frac{1}{p-1}}$ as $r\to+\infty$. Suppose by contradiction that $u_s(r)\to0$ as $r\to+\infty$. It follows from \cite[Lemma~2.3]{BK} that $u_{s}(r)\sim u_{\infty}e^{-\nu r}$ as $r\to+\infty$. Noticing that $\nu>N-1$ since $\lambda<0$, we have
\begin{align*}
	|((\sinh r)^{N-1}u_s')'|\le(\sinh r)^{N-1}(|\lambda| u_s+u_s^{p})\in L^{1}(1,+\infty),
\end{align*}
which implies that there exists $C'\le0$ such that $\lim_{r\to+\infty}(\sinh r)^{N-1}u_s'(r)=C'$. If $C'<0$, then $u_s'(r)\sim C_1e^{(1-N)r}$ and $u_{s}(r)\sim C_2e^{(1-N)r}$ as $r\to+\infty$, where $C_1<0$ and $C_2>0$. This yields $u_s(r)e^{\nu r}\to+\infty$ as $r\to+\infty$, which is a contradiction. Thus, $C'=0$. Integrating \eqref{eq:radial} from $r$ to $+\infty$, we have 
\begin{align*}
	(\sinh r)^{N-1}u_s'(r)=\int_{r}^{+\infty}(\sinh\rho)^{N-1}(\lambda u_s+u_s^{p})\,d\rho,
\end{align*}
which yields $u_{s}'(r)\sim -Ce^{-\nu r}$ as $r\to+\infty$.
Thus, 
\begin{align*}
	(\sinh r)^{N}u_s'^2(r)\sim Ce^{(N-2\nu) r}\to0\qquad\text{as }r\to+\infty.
\end{align*}
	Hence, all the boundary terms at $r_2 \to \infty$ in \eqref{pohozaev-identity} vanish. However, since $p>\frac{N+2}{N-2}$ and $\lambda<0$, the right-hand side of \eqref{pohozaev-identity} is strictly negative, whereas the left-hand side tends to $0$. This yields a contradiction.

\textit{Case 2.} $\lambda=0$. Then $u_{s}(r)\to0$ as $r\to+\infty$. From \cite[Lemma~2.5]{BK}, $u_s$ satisfies either 
$u_s(r)\sim u_{\infty}e^{(1-N)r}$ or $u_s(r)\sim \left(\frac{N-1}{p-1}\right)^{1/(p-1)}r^{-1/(p-1)}$ as $r\to+\infty$. If the former holds, it follows from \cite[Lemma~2.5]{BK} that $\frac{u_s'(r)}{u_s(r)}\to-(N-1)$ as $r\to+\infty$. This yields $u_s'(r)\sim (1-N)u_{\infty}e^{(1-N)r}$ as $r\to+\infty$. Consequently, the left-hand side of \eqref{pohozaev-identity} tends to $0$ as $r_1\to0$ and $r_2\to+\infty$, while the right-hand side is strictly negative, which is absurd. Hence, $u_s(r)\sim \left(\frac{N-1}{p-1}\right)^{1/(p-1)}r^{1/(p-1)}$ as $r\to+\infty$.

\textit{Case 3.} $0<\lambda\le\frac{N(N-2)}{4}$. Then $u_{s}(r)\to0$ as $r\to+\infty$. It follows from \cite[Lemma~2.3]{BK} that $u_{s}$ satisfies \eqref{24} or $u_{s}(r)\sim\bar{u}_\infty e^{-\mu r}$ as $r\to+\infty$. If $\lambda\in\left(0,\frac{N(N-2)}{4}\right)$, applying Lemma~\ref{lem4} and an argument similar to that in Case $1$, we obtain $u_{s}(r)\sim \bar{u}_\infty e^{-\mu r}$ as $r\to+\infty$. When $\lambda=\frac{N(N-2)}{4}$, suppose by contradiction that $u_s(r)\sim u_\infty e^{-\frac{N}{2}r}$ as $r\to+\infty$. Applying Lemma~\ref{lem4} again, we obtain 
\begin{align*}
	(\sinh r)^Nu_s'^2(r)\sim\frac{N^2u_{\infty}^2}{2^{N+2}},\quad (\sinh r)^{N}\coth r\,u_{s}(r)u_{s}'(r)\sim-\frac{Nu_{\infty}^2}{2^{N+1}}\qquad\text{as }r\to+\infty.
	\end{align*}
Moreover, we have
\begin{align*}
	(\sinh r)^{N}u_{s}^2(r)\sim\frac{u_{\infty}^2}{2^N},\quad (\sinh r)^{N}u_{s}^{p+1}(r)\to0\qquad\text{as }r\to+\infty.
	\end{align*} 
	Therefore, as $r_{1}\to0^+$ and $r_{2}\to+\infty$, the left-hand side of \eqref{pohozaev-identity} tends to $\frac{u_{\infty}^2}{2^{N+1}}>0$. However, since $p>\frac{N+2}{N-2}$, the right-hand side is strictly negative. This contradiction yields $u_{s}(r)\sim\bar{u}_{\infty}e^{-\frac{N-2}{2}r}$ as $r\to+\infty$. 

The proof of Theorem~\ref{thm4} is completed.\qed

\section{Proof of Theorems \ref{thm2} and \ref{thm3}}\label{S4}
In this section, we investigate the positive radial singular solutions to \eqref{eq1} in the critical case $p=\frac{N+2}{N-2}$ and $\lambda=\frac{N(N-2)}{4}$, namely the problem 
\begin{equation}\label{critical}
\begin{cases}
    -u''-(N-1)\coth r\,u'-\frac{N(N-2)}{4}u=u^{\frac{N+2}{N-2}},& r>0,\\
    u(r)\to+\infty\qquad\text{as }r\to0^+.
    \end{cases}
\end{equation}
  
  We consider the transformation
\begin{align*}
v(t)=(\sinh r)^{\frac{N-2}{2}}u(r),\qquad  t=\ln\tanh\frac{r}{2}.
\end{align*}
Then $v$ satisfies the autonomous equation
\begin{equation}\label{1}
v''-\frac{(N-2)^2}{4}v+v^{\frac{N+2}{N-2}}=0,\qquad t\in(-\infty,0).
\end{equation}
Observe that $ v_0=\left(\frac{N-2}{2}\right)^{(N-2)/2}$ is a constant solution to \eqref{1}. Consequently,
 \begin{align*}
     U_s(r)=\left(\frac{N-2}{2\sinh r}\right)^{\frac{N-2}{2}}
 \end{align*}
is an explicit positive singular solution to \eqref{critical}.

Define the energy functional 
\begin{align*}
    \mathcal{E}_{v}(t)=\frac{1}{2}v'^2(t)+F(v(t)),
\end{align*}
where 
\begin{align*}
    F(v)=\int_{0}^{v}\biggl(-\frac{(N-2)^2}{4}\rho+\rho^{\frac{N+2}{N-2}}\biggr)\,d\rho=-\frac{(N-2)^2}{8}v^2+\frac{N-2}{2N}v^{\frac{2N}{N-2}}.
\end{align*}
For every solution $v$ of \eqref{1}, we have
\begin{equation}\label{45}
    \frac{d}{dt}\mathcal{E}_{v}(t)=v'\left(v''-\frac{(N-2)^{2}}{4}v+v^{\frac{N+2}{N-2}}\right)=0,
\end{equation}
which means that $\mathcal{E}_{v}$ is constant along any solution of \eqref{1}.

We first prove the existence of positive solutions of \eqref{1}.
\begin{lem}\label{lem1}
	Let $N\ge3$, and let $v$ be the maximal solution of \eqref{1}
	satisfying
	\begin{align*}
	v(t_0)=\beta>0,\qquad v'(t_0)=\vartheta
	\end{align*}
	for some $t_0<0$, where $\beta$ and $\vartheta$ satisfy
	\begin{equation}\label{46}
		F(v_0)<E:=\frac{1}{2}\vartheta^2+F(\beta)<0.
	\end{equation}
Then $v$ extends to a positive nonconstant periodic solution on $\RR$.
\end{lem}

\begin{proof}
A simple calculation shows that $F$ is strictly decreasing on $(0,v_0)$, strictly increasing
on $(v_0,\infty)$, and $F(0)=F(v^{*})=0$, where $v^{*}=(N(N-2)/4)^{(N-2)/4}$. Moreover, $F(v_0)<E<0$. Then the equation $F(\xi)=E$ has exactly two positive roots
$v_-$ and $v_+$ with $0<v_-<v_0<v_+<v^{*}$. Since $\mathcal{E}_{v}(t)\equiv E$, we have $v_-\le v(t)\le v_+$ on the maximal interval of existence. Moreover,
\begin{align*}
v'^2(t)=2\bigl(E-F(v(t))\bigr)\le2\bigl(E-F(v_0)\bigr).
\end{align*}
Thus both $v$ and $v'$ are bounded, so the solution $v$ can be extended to all of $\RR$.

We next show that $v$ cannot be strictly monotone on $(-\infty,t_0]$. Suppose by contradiction that $v$ is strictly monotone on $(-\infty,t_0]$. Since $v$ is bounded, there exists $v_{\infty}\in[v_-,v_+]$ such that $\lim_{t\to-\infty}v(t)=v_{\infty}$. This yields $\lim_{t\to-\infty}v'^2(t)=2\bigl(E-F(v_{\infty})\bigr)\ge0$. If $E-F(v_{\infty})>0$, together with the fact that $v$ is strictly monotone on $(-\infty,t_0]$, we obtain $|v'(t)|\ge C>0$ for $t$ sufficiently negative. This contradicts the boundedness of $v$. Thus, $\lim_{t\to-\infty}v'(t)=0$. Equation \eqref{1} yields
\begin{align*}
	\lim_{t\to-\infty}v''(t)=\frac{(N-2)^2}{4}v_{\infty}-v_{\infty}^{\frac{N+2}{N-2}}.
\end{align*} 
Since $\lim_{t\to-\infty}v'(t)=0$, the above limit must be zero, which implies $v_{\infty}=v_0$. It follows that $E=F(v_0)$, which is impossible. Hence, there exists $t_1<t_0$ such that $v'(t_1)=0$. Applying the same argument on $(-\infty,t_1]$, we can find $t_2<t_1$ such that $v'(t_2)=0$. Set $\widetilde v(t)=v(2t_1-t)$. Then $\widetilde v$ satisfies \eqref{1}. Moreover, $\widetilde v(t_1)=v(t_1)$ and $\widetilde v'(t_1)=-v'(t_1)=0$. By uniqueness, $\widetilde v\equiv v$. Now set $\overline v(t)=v\bigl(2t_1-2t_2+t\bigr)$. Using the above symmetry, we obtain $\overline v(t_2)=v(2t_1-t_2)=v(t_2)$ and $\overline v'(t_2)=v'(2t_1-t_2)=-v'(t_2)=0$. By uniqueness again, $\overline v\equiv v$. Therefore, $2(t_1-t_2)$ is a period of $v$. 

Finally, since $E>F(v_0)$, the solution cannot be the constant solution $v\equiv v_0$. Therefore, $v$ is a nonconstant periodic solution.
\end{proof}

\textbf{Proof of Theorem~\ref{thm2}.} By Lemma~\ref{lem1}, there exists a family of nonconstant positive periodic solutions of \eqref{1} satisfying $0<v(t)<v^{*}$ for all $t\in\RR$, where $v^{*}=(N(N-2)/4)^{(N-2)/4}$. Consequently, each such periodic solution generates a positive
radial singular solution of \eqref{critical} defined by $u_s(r)=(\sinh r)^{\frac{2-N}{2}}v(t)$, where $t=\ln\tanh\frac{r}{2}$. Since $\lim_{t\to0^{-}}v(t)=v(0)>0$, we have
\begin{align*}
	u_s(r)=(\sinh r)^{\frac{2-N}{2}}v(t)\sim 2^{\frac{N-2}{2}}v(0)e^{\frac{2-N}{2}r}\qquad\text{as }r\to+\infty.
	\end{align*}
Moreover, since $F(v_0)<E<0$, there exist $v_-$ and $v_{+}$ with $0<v_-<v_0<v_+<v^*$ such that $v(t)\in[v_-,v_+]$ for all $t\in\RR$. Since $v\in C^2(\RR)$ is nonconstant and periodic, it attains its minimum and maximum values. Thus, \eqref{47} holds. The proof of Theorem~\ref{thm2} is completed.\qed

\textbf{Proof of Theorem~\ref{thm3}.}
For $t\in(-\infty,0)$, set
\begin{align*}
	r(t)=\ln\frac{1+e^t}{1-e^t},
	\qquad
	v(t)=(\sinh r(t))^{(N-2)/2}u(r(t)).
\end{align*}
Then $v(t)>0$ for $t\in(-\infty,0)$ and satisfies \eqref{1}. Since $\mathcal{E}_{v}$ is constant along any solution of \eqref{1}, there exists $E\in\RR$ such that $\mathcal{E}_v(t)\equiv E$ for any $t\in(-\infty,0)$.

If $v\equiv v_0$, then $u=U_s$. In what follows, we assume that $v$ is nonconstant.

Suppose first that $v'$ has at most one zero. Then $v$ is monotone for $t<0$ sufficiently negative. We will show that this case cannot occur. Since $F(v)\le E$ and $F(v)\to+\infty$ as $v\to+\infty$, $v$ is bounded on $(-\infty,0)$. Consequently, there exists a constant $c\ge0$ such that $v(t)\to c$ as $t\to-\infty$. As in the proof of Lemma~\ref{lem1}, we obtain $v'(t)\to0$ as $t\to-\infty$. Equation \eqref{1} yields $v''(t)\to -F'(c)$ as $t\to-\infty$, which implies $c=v_0$ or $c=0$.

\textbf{Case 1.} $c=v_0$. Then $E=F(v_0)$. Since $F(v)\ge F(v_0)$ for $v>0$, we get $v\equiv v_0$, which is a contradiction.

\textbf{Case 2.} $c=0$. Then $E=0$. Moreover, $v>0$ on $(-\infty,0)$ and $v(t)\to0$ as $t\to-\infty$. Thus, for $t<0$ sufficiently negative, $v'(t)>0$ and
\begin{align*}
	v'=\sqrt{-2F(v)}=v\sqrt{\frac{(N-2)^2}{4}-\frac{N-2}{N}v^{\frac{4}{N-2}}}.
\end{align*}
This also yields 
\begin{align*}
	v'^2=\frac{(N-2)^2}{4}v^2-\frac{N-2}{N}v^{\frac{2N}{N-2}}.
\end{align*}
Hence, there exists $t_0\in\RR$ such that for all $t<0$ sufficiently negative, 
\begin{align*}
	v(t)=v^*\left(\cosh(t-t_0)\right)^{-\frac{N-2}{2}},
\end{align*}
where $v^{*}=(N(N-2)/4)^{(N-2)/4}$. Since the right-hand side defines a positive solution to \eqref{1} on $\RR$, by uniqueness $v$ can be extended to $\RR$. It follows that 
\begin{align}\label{eq:regular}
    u(r)=\frac{v^{*}(2\cosh t_0)^{-\frac{N-2}{2}}}{\Bigl(\bigl(\cosh\frac{r}{2}\bigr)^2-\frac{1}{1+e^{2t_0}}\Bigr)^{\frac{N-2}{2}}},\qquad r>0.
\end{align}
Thus, $\lim_{r\to0^+}u(r)=v^*e^{-\frac{N-2}{2}t_0}$, which contradicts the assumption that the singularity at $Q$ is nonremovable. In fact, the solution in \eqref{eq:regular} is precisely the $k=1$ profile classified in \cite[Theorem~1.1(2)]{LLW}.

Now we suppose that $v'$ has at least two zeros on $(-\infty,0)$. Let $t_1<t_2<0$ satisfy $v'(t_1)=v'(t_2)=0$. If $v(t_1)=v(t_2)$, by uniqueness, $v$ is periodic with period $t_2-t_1$. If $v(t_1)\neq v(t_2)$, $F(v)=E$ has two distinct positive roots. This implies $F(v_0)<E<0$. Hence, by Lemma~\ref{lem1}, $v$ is a
positive nonconstant periodic solution on $\RR$. 

Let $b=\inf_{\RR}v$. Then $b\in(0,v_0)$. Choose $t_0\in\RR$ such that $v(t_0)=b$, and so $v'(t_0)=0$. Let $v_b$ be the solution of \eqref{1} normalized by $v_b(0)=b$ and $v_b'(0)=0$. By the uniqueness of the initial-value problem, $v(t)=v_b(t-t_0)$ for all $t\in\RR$. The periodicity of $v_b$ implies that there exists
a unique $T\in[0,T_b)$ such that $v(t)=v_b(t+T)$ for any $t\in\RR$. By the definition of $v$, we obtain
\begin{align*}
u(r)
=
(\sinh r)^{-\frac{N-2}{2}}
v_b\left(\ln\tanh\frac r2+T\right),
\qquad r>0.
\end{align*}
The proof of Theorem~\ref{thm3} is completed.\qed

\section{Appendix}
In this section, we introduce some results for the corresponding equation in $\RR^{N}$ that are used in Section~2. A solution $\widetilde{w}_\alpha$ of \eqref{linear2} is called a slowly decaying solution if $\widetilde{w}_\alpha(\tau)>0$ on $(0,+\infty)$ and $\tau^{N-2}\widetilde{w}_\alpha(\tau)\to+\infty$ as $\tau\to+\infty$. Yanagida and Yotsutani \cite{YY} obtained the existence of slowly decaying solutions to \eqref{linear2} under the following conditions:
\begin{itemize}
	\item[($\mathrm{K_1}$)] $K\in C((0,+\infty))$, $K(\tau)\ge0$ and $K(\tau)\not\equiv0$ on $(0,+\infty)$, with $\tau K(\tau)\in L^1(0,1)$,
	 \item[($\mathrm{K_2}$)] $\tau^{N-1-(N-2)p}K\in L^1(1,+\infty)$.
\end{itemize} 
Let $G(\tau)$ and $H(\tau)$ be functions defined by
\begin{alignat*}{2}
	G(\tau)&=\frac{1}{p+1}\tau^{N}K(\tau)-\frac{N-2}{2}\int_{0}^{\tau}\rho^{N-1}K(\rho)\,d\rho,\\
	H(\tau)&=\frac{1}{p+1}\tau^{2-(N-2)p}K(\tau)-\frac{N-2}{2}\int_{\tau}^{+\infty}\rho^{1-(N-2)p}K(\rho)\,d\rho.
	\end{alignat*}
Define
\begin{alignat*}{2}
	\tau_{G}&=\inf\{\tau\in(0,+\infty):G(\tau)<0\},\\
	\tau_{H}&=\sup\{\tau\in(0,+\infty):H(\tau)<0\}.
	\end{alignat*}
	\begin{lem}\label{lem8}\cite[Theorem~2]{YY}
		Assume that $N>2$, $p>1$, and $K$ satisfies $(\mathrm{K_1})$--$(\mathrm{K_2})$. If $\liminf_{\tau\to+\infty}G(\tau)<0$, then there exists $\alpha_*>0$ such that $\widetilde{w}_\alpha$ is a slowly decaying solution of \eqref{linear2} for every $\alpha\in(0,\alpha_*)$.
	\end{lem}
	\begin{lem}\label{lem9}\cite[Lemma~6]{YY}
		Let $N>2$ and $p>1$. Suppose that $K$ satisfies $(\mathrm{K_1})$--$(\mathrm{K_2})$ and
		\begin{align*}
			K(\tau)=A\tau^{l}+o(\tau^{l})\qquad\text{as }\tau\to+\infty
		\end{align*}  
		for some $A>0$ and $l<(N-2)p-N$.
		\begin{enumerate}[label=\textnormal{(\roman*)}]
			\item If $l<\frac{(N-2)p-(N+2)}{2}$, then $\limsup_{\tau\to+\infty}G(\tau)<0$ and $\tau_{H}<+\infty$.
			\item If $l>\frac{(N-2)p-(N+2)}{2}$, then $\liminf_{\tau\to+\infty}G(\tau)>0$ and $\tau_{H}=+\infty$.
		\end{enumerate}
	\end{lem}
	
Li \cite{L1} studied the asymptotic behavior at infinity of solutions to \eqref{linear2} and proved the following results. For a real-valued function $f$, set $f^{+}=\max\{f,0\}$ and $f^{-}=\max\{-f,0\}$.
\begin{lem}\label{lem11}\cite[Theorem~1]{L1}
		Let $N\ge3$, $p>1$ and let $\widetilde{w}$ be a positive radial solution of \eqref{linear2}. Assume that $K$ satisfies
		\begin{enumerate}[label=\textnormal{(\roman*)}]
		\item $K(\tau)\ge0$ in $\tau\ge0$, $\lim_{\tau\to+\infty}\tau^{-l}K(\tau)=k_\infty>0$ where $l\ge-2$, $K$ is differentiable near infinity and $\left[\frac{d}{d\tau}(\tau^{-l}K(\tau))\right]^{+}\in L^1$, if $0<\frac{2+l}{p-1}<\frac{N-2}{2}$, or
		\item  $K(\tau)\ge0$ in $\tau\ge0$, $\lim_{\tau\to+\infty}\tau^{-l}K(\tau)=k_\infty>0$ where $l\ge-2$, $K$ is differentiable near infinity and $\left[\frac{d}{d\tau}(\tau^{-l}K(\tau))\right]^{-}\in L^1$, if $\frac{N-2}{2}<\frac{2+l}{p-1}<N-2$,
		\end{enumerate}
		then
		\begin{align*}
		\lim_{\tau\to+\infty}\tau^{\frac{2+l}{p-1}}\widetilde{w}(\tau)=:\widetilde{w}_\infty=\begin{cases}
			\biggl[\frac{\frac{2+l}{p-1}\bigl(N-2-\frac{2+l}{p-1}\bigr)}{k_\infty}\biggr]^{\frac{1}{p-1}}\quad\text{or}\\
			0.
		\end{cases}
		\end{align*}
		Furthermore, if $\widetilde{w}_\infty=0$, then 
		\begin{align*}\lim_{\tau\to+\infty}\tau^{N-2}\widetilde{w}(\tau)\end{align*} exists and is finite and positive.
	\end{lem}
\begin{lem}\label{lem13}\cite[Theorem~2]{L1}
	Let $N\ge3$, $p>1$ and let $\widetilde{w}$ be a positive radial solution of \eqref{linear2}. Assume that $K$ satisfies $K(\tau)\ge0$ in $\tau\ge0$, $\lim_{\tau\to+\infty}\tau^{2}K(\tau)=k_\infty>0$, $K$ is differentiable near infinity and $\left[\frac{d}{d\tau}(\tau^{2}K(\tau))\right]^{+}\in L^1$. Then $\lim_{\tau\to+\infty}(\ln\tau)^{1/(p-1)}\widetilde{w}(\tau)$ always exists and 
	\begin{align*}
		\lim_{\tau\to+\infty}(\ln\tau)^{1/(p-1)}\widetilde{w}(\tau)=:\widetilde{w}_\infty=\begin{cases}
			\Bigl[\frac{N-2}{(p-1)k_\infty}\Bigr]^{\frac{1}{p-1}}\quad\text{or}\\
			0.
			\end{cases}
	\end{align*}
	Furthermore, if $\widetilde{w}_\infty=0$, then 
	\begin{align*}\lim_{\tau\to+\infty}\tau^{N-2}\widetilde{w}(\tau)\end{align*} exists and is finite and positive.
\end{lem}	
	
\subsection*{Acknowledgemnet} The research of the authors is partial supported by NSFC (12671132).


\begin{thebibliography}{99}
	
\bibitem{A} P. Aviles, Local behavior of solutions of some elliptic equations, {\it Comm. Math. Phys.} {\bf 108} (1987), 177--192.

\bibitem{BK} C. Bandle and Y. Kabeya, On the positive, ``radial'' solutions of a semilinear elliptic equation in $\HH^N$, {\it Adv. Nonlinear Anal.} {\bf 1} (2012), 1--25.

\bibitem{BK1} C. Bandle and Y. Kabeya, Erratum: On the positive, ``radial'' solutions of a semilinear elliptic equation in $\HH^N$ [Adv. Nonlinear Anal. 1 (2012), no. 1, 1--25], {\it Adv. Nonlinear Anal.} {\bf 2} (2013), 147--150.

\bibitem{BGGV} M. Bonforte, F. Gazzola, G. Grillo and J. L. V\'azquez, Classification of radial solutions to the Emden--Fowler equation on the hyperbolic space, {\it Calc. Var. Partial Differential Equations} {\bf 46} (2013), 375--401.

\bibitem{CGS} L. Caffarelli, B. Gidas and J. Spruck, Asymptotic symmetry and local behavior of semilinear elliptic equations with critical Sobolev growth, {\it Comm. Pure Appl. Math.} {\bf 42} (1989), 271--297.

\bibitem{CL} E. Coddington and N. Levinson, {\it Theory of ordinary differential equations}, McGraw-Hill Book Co., Inc., New York-Toronto-London, 1955.

\bibitem{DPGW} A. DelaTorre, M. del Pino, M. Gonz\'alez and J. Wei, Delaunay-type singular solutions for the fractional Yamabe problem,
{\it Math. Ann.} {\bf 369} (2017), 597--626.

\bibitem{FK} R. Frank and T. K\"onig,
Classification of positive singular solutions to a nonlinear biharmonic equation with critical exponent,
{\it Anal. PDE} {\bf 12} (2019), 1101--1113.

\bibitem{GS} B. Gidas and J. Spruck, Global and local behavior of positive solutions of nonlinear elliptic equations, {\it Comm. Pure Appl. Math.} {\bf 34} (1981), 525--598.

\bibitem{GHWW} Z. Guo, X. Huang, L. Wang and J. Wei,
On Delaunay solutions of a biharmonic elliptic equation with critical exponent,
{\it J. Anal. Math.} {\bf 140} (2020), 371--394.

\bibitem{H} S. Hasegawa, Singular solutions of semilinear elliptic equations with supercritical growth on Riemannian manifolds, {\it NoDEA Nonlinear Differential Equations Appl.} {\bf 31} (2024), Paper No. 39.

\bibitem{H1} H. He, Supercritical elliptic equation in hyperbolic space, {\it J. Partial Differ. Equ.} {\bf 28} (2015), 120--127.

\bibitem{HLW} C. Hsia, C. Lin and Z. Wang, Asymptotic symmetry and local behaviors of solutions to a class of
anisotropic elliptic equations, {\it Indiana Univ. Math. J.} {\bf 60} (2011), 1623--1654.

\bibitem{JX} T. Jin and J. Xiong, Asymptotic symmetry and local behavior of solutions of higher order conformally invariant equations with isolated singularities, {\it Ann. Inst. H. Poincar\'e{} C Anal. Non Lin\'eaire} {\bf 38} (2021), 1167--1216.

\bibitem{L1} Y. Li, Asymptotic behavior of positive solutions of equation $\Delta u+K(x)u^p=0$ in $\RR^n$, {\it J. Differential Equations} {\bf 95} (1992), 304--330.

\bibitem{LLW} J. Li, G. Lu and J. Wang, The method of moving spheres on the hyperbolic space and the classification of solutions and the prescribed $Q$-curvature problem, {\it Adv. Math.} {\bf 482} (2025), Paper No. 110606.

\bibitem{L4} P.-L. Lions, Isolated singularities in semilinear problems, {\it J. Differential Equations} {\bf 38} (1980), 441--450.

\bibitem{MS} G. Mancini and K. Sandeep, On a semilinear elliptic equation in $\HH^N$, {\it Ann. Sc. Norm. Super. Pisa Cl. Sci. (5)} {\bf 7} (2008), 635--671.

\bibitem{NS1} W.-M. Ni and J. Serrin, Nonexistence theorems for singular solutions of quasilinear partial differential equations, {\it Comm. Pure Appl. Math.} {\bf 39} (1986), 379--399.

\bibitem{SZ} J. Serrin and H. Zou, Classification of positive solutions of quasilinear elliptic equations, {\it Topol. Methods Nonlinear Anal.} {\bf 3} (1994), 1--25.

\bibitem{WCCK} Y. Wu, Z. Chen, J. Chern and Y. Kabeya, Existence and uniqueness of singular solutions for elliptic equation on the hyperbolic space, {\it Commun. Pure Appl. Anal.} {\bf 13} (2014), 949--960.

\bibitem{YY} E. Yanagida and S. Yotsutani, Existence of positive radial solutions to $\Delta u+K(|x|)u^p=0$ in $\RR^n$, {\it J. Differential Equations} {\bf 115} (1995), 477--502.



\end{thebibliography}
\end{document}